%% file: writeup-v2.tex
\documentclass[11pt]{article}

\usepackage{sn-preamble}

\newcommand{\diam}{\operatorname{diam}}

\renewcommand{\S}{\mathbb{S}}
\newcommand{\dG}{\mathrm{dist}_{\mathrm{G}}}

\newcommand{\Conv}{\mathrm{Conv}}
\newcommand{\dtv}{\mathrm{d}_{\mathrm{TV}}}
\newcommand{\dkl}{\mathrm{d}_{\mathrm{KL}}}
\newcommand{\newalpha}{\alpha(r,\eps)}
\newcommand{\OPT}{\mathrm{OPT}}

\title{Optimal Sparsification of Gaussian Processes}
\author{%
	Shivam Nadimpalli \\
	\emph{Massachusetts Institute of Technology} \\
	\href{mailto:shivamn@mit.edu}{\texttt{shivamn@mit.edu}}
}
\date{\today}

\begin{document}

\pagenumbering{gobble}

\maketitle 

\begin{abstract}
We prove an optimal dimension-free sparsification theorem for suprema of centered Gaussian processes. 
Given a bounded set $T\sse\R^n$, we show that the supremum of the canonical Gaussian process on $T$ can be $L^2$-approximated by the supremum of a shifted subprocess indexed by only $\exp(O(1/\eps^2))$ points, with error at most $\eps$ times the Gaussian width of $T$. 
In particular, the size of the approximating process is independent of both the ambient dimension and the cardinality of the original index set. 

This improves a recent sparsification theorem of De, Nadimpalli, O'Donnell, and Servedio (2026) by an exponential factor, and we show that the dependence on $\eps$ is tight up to constants in the exponent. 
As consequences, we obtain an exponentially improved junta theorem for norms over Gaussian space and sharpen results on learning, property testing, and polyhedral approximation of convex sets under the Gaussian measure. The proof is based on an interpolation argument that combines Sudakov's minoration with the Brascamp--Lieb inequality. 
\end{abstract}

\newpage 
\setcounter{page}{1}
\pagenumbering{arabic}

\input{sections/introduction}
\input{sections/preliminaries}
\input{sections/improved-sparsification}
\input{sections/applications}
\input{sections/lower-bounds}

\section*{Acknowledgements} 
 
S.N.~is partially supported by Elchanan Mossel's Vannevar Bush Faculty Fellowship ONR-N00014-20-1-2826.
The author is grateful to Anindya De, Larry Guth, Amit Rajaraman, and Rocco Servedio for encouragement, and to Jun-Ting (Tim) Hsieh for suggesting \Cref{fig:hexagon-shift}. 

\section*{AI Disclosure} 

We used large language models (LLMs) to assist with editing, drafting TikZ code for \Cref{fig:hexagon-shift}, literature search, and manuscript preparation.  
In particular, GPT~5.4~and~5.5 pointed us to the Zamir--Feder entropy inequality, used in \Cref{lem:zamir-feder-projection-entropy}; to \cite[Problem~3.19(d)]{rvh-notes}, used in the proof of \Cref{lemma:LSI-strongly-log-concave}; and answered background questions about differential equations.
All mathematical claims, proofs, and references were checked by the author.  

\bibliography{allrefs}
\bibliographystyle{alphaurl}

\appendix

\input{sections/technical-lemmas}
\end{document}

%% file: sections/introduction.tex
\section{Introduction}
\label{sec:intro}

The study of suprema of Gaussian processes has a long history in probability theory, with influential applications across modern analysis, geometry, and theoretical computer science~\cite{rvh-notes,Nelson16,vershynin2018high,talagrand2022upper}. 
Our main result, \Cref{thm:GP-sparsification}, is an optimal dimension-free sparsification theorem for these random variables. 
We begin by recalling the background on Gaussian processes that motivates our result. 

\paragraph{Suprema of Gaussian Processes.} 

A Gaussian process is a collection $\{\bX_t\}_{t\in T}$ of (possibly infinitely many) random variables, indexed by a set $T$, such that every finite subcollection is jointly Gaussian.  
The object of interest in this paper is the supremum of such a process, namely the random variable $\sup_{t\in T} \bX_t$. 
These random variables arise naturally in high-dimensional settings, where one often needs uniform control over large families of correlated random quantities. 
Moreover, many familiar mathematical objects can themselves be expressed as suprema, including support functions of convex bodies, empirical risk minimization objectives, and operator norms of random matrices~\cite[Section~1.2.2]{rvh-notes}. 

A central achievement of the theory is an essentially complete understanding of the expected supremum of a Gaussian process in terms of the \emph{geometry} of its index set. 
To explain this connection, recall that every centered Gaussian process can be represented, without loss of generality, as a \emph{canonical} Gaussian process~\cite[Section~7.1.2]{vershynin2018high}. 
In finite dimensions, this means a process of the form 
\begin{equation} \label{eq:canonical-GP}
	\cbra{\bX_t := \bg \cdot t}_{t\in T}~\text{where}~T\sse\R^n~\text{and}~\bg\sim N(0, I_n)\,.
\end{equation}
Here $N(0, I_n)$ denotes the $n$-dimensional standard Gaussian distribution.  
In this representation, the covariance structure is encoded by the Euclidean geometry of $T$: for $s,t\in T$, we have $\Ex[\bX_s\bX_t] = s\cdot t$. 
Thus the geometry of $T$ governs the behavior of the process. 
Classical results such as Dudley's entropy bound, Sudakov's minoration, and the Fernique--Talagrand theory of majorizing measures show that this geometric viewpoint is essentially complete at the level of the expected supremum: the quantity $\E[\sup_{t\in T} \bX_t]$ is characterized, up to universal constant factors, by the (multi-scale) geometry of the index set $T$~\cite{rvh-notes,talagrand2022upper}. 

We focus exclusively on processes of the form \eqref{eq:canonical-GP}. 
For such a process indexed by a set $T\sse \R^n$, the expected supremum is called the \emph{Gaussian width} of $T$~\cite[Chapter~7]{vershynin2018high}, and we write
\[
	w(T) := \Ex \sbra{\sup_{t\in T} \bg\cdot t}
\]
for this quantity. 
The Gaussian width $w(T)$ is both the natural scale of the process and a fundamental complexity measure of the index set, and all of our approximation guarantees are measured relative to it. 

\paragraph{Sparsification of Gaussian Processes.} 

A natural question is whether this geometric viewpoint can be upgraded from the level of expectation to the level of the random variable itself. 
Namely, can $\sup_{t\in T}\bX_t$ be approximated (say in $L^1$ or $L^2$; see \Cref{subsec:prelims-gaussian-rvs} for definitions), with error measured at scale $w(T)$, by a process indexed by far fewer points? 
A dimension-free sparsification theorem of this form was obtained by De, Nadimpalli, O'Donnell, and Servedio~\cite[Theorem~1]{de2026sparsifying}, who proved that for every bounded $T\sse \R^n$ and every $\varepsilon>0$, there is a subset $S\sse T$ and a collection of non-negative \emph{shifts} $\{c_s\geq 0\}_{s\in S}$ such that 
\[
    |S| \leq \exp\exp\pbra{O\pbra{\frac{1}{\eps}}}
    \qquad\text{and}\qquad 
    \left\|
    \sup_{t\in T} \bg\cdot t
    -
    \max_{s\in S} \cbra{\bg\cdot s + c_s}
    \right\|_{L^1}
    \leq
    \eps\cdot w(T)\,.
\] 

Thus the supremum of the original process can be $L^1$-approximated, at scale $w(T)$, by the supremum of a shifted subprocess whose size is independent of both the ambient dimension $n$ and the cardinality of $T$. 
Moreover, by Markov's inequality, the $L^1$ guarantee immediately yields a high-probability approximation: for every $\eta>0$, 
\[
	\Prx_{\bg\sim N(0, I_n)}\left[
	\left|
    \sup_{t\in T} \bg\cdot t
    -
    \max_{s\in S} \cbra{\bg\cdot s + c_s}
	\right|
	>
	\frac{\eps\cdot w(T)}{\eta}
	\right]
	\leq \eta\,.
\]
In particular, to approximate $\sup_{t\in T}\bg\cdot t$ to additive error $\delta$ with probability at least $1-\eta$, it suffices to choose $\eps \asymp \eta\cdot\delta/w(T)$, and so the size of the sparsified process depends only on $w(T)$, $\delta$, and $\eta$, and not on the ambient dimension $n$ or on the cardinality of $T$. 
This reinforces the role of Gaussian width as a basic complexity measure of the index set: beyond setting the natural scale of the process, it governs how large a shifted subprocess is needed to approximate the supremum as a random variable. 

We note that the shifts $\{c_s\}$ are necessary if one requires $S\sse T$ while keeping $|S|$ independent of the ambient dimension; see \cite[Example~32]{de2026sparsifying}. 
On the other hand, if one drops the requirement that the approximating index set be a subset of $T$, then with some additional work one can obtain a centered approximator to the original supremum on a new index set~\cite[Corollary~23]{de2026sparsifying}.  

The proof of \cite{de2026sparsifying} constructs a multi-scale partition of the index set $T$ using Talagrand's celebrated majorizing measures theorem~\cite{talagrand2022upper}. 
The same work also proves a lower bound of $\exp\big(\widetilde{\Omega}(1/\sqrt{\varepsilon})\big)$ on the size of any such approximator; see \cite[Corollary~35]{de2026sparsifying}. 
This leaves open the quantitative question of whether the double-exponential dependence on $\eps^{-1}$ is inherent, or merely an artifact of the multi-scale construction. 

\paragraph{This Work: Optimal Bounds.} 

Our main result answers this question in the negative. 
For every bounded index set $T\sse \R^n$, we construct a shifted subprocess indexed by at most $\exp(O(1/\eps^2))$ points whose supremum approximates the original supremum in $L^2$ to error at most $\eps\cdot w(T)$. 
This improves the bound of \cite{de2026sparsifying} from double-exponential to single-exponential, while also strengthening the approximation guarantee from $L^1$ to $L^2$. 

\begin{restatable}{theorem}{gpsparsificationthm}
\label{thm:GP-sparsification}
	Suppose $T \sse \R^n$ is bounded and let $\eps > 0$. 
	There exists $S \sse T$ and non-negative constants $\{c_s \geq 0\}_{s\in S}$ such that 
	\[
		|S| \leq \exp\pbra{O\pbra{\frac{1}{\eps^2}}}
		\qquad\text{and}\qquad 
		\vabs{
			\sup_{t\in T}\bg\cdot t - \max_{s \in S} \cbra{\bg\cdot s + c_s} 
		}_{L^2} 
		\leq 
		\eps\cdot w(T)\,. 
	\]
\end{restatable} 

We briefly highlight the main idea of the proof, deferring a detailed technical overview and comparison with \cite{de2026sparsifying} to \Cref{subsec:technical-overview}.   
We take $S\sse T$ to be an $O(\eps\cdot w(T))$-net of $T$, which by Sudakov's minoration (\Cref{lem:sudakov}) has size $\exp(O(1/\eps^2))$.  
The shifts $\{c_s\}_{s\in S}$ compensate for the fine-scale fluctuations lost when passing from $T$ to the net $S$.  
The correct shift for each net point turns out to depend on where that point is active in the maximum, but these ``active regions'' themselves depend on the shifts.  
We resolve this circularity by constructing the shifts dynamically along an interpolation from $T$ to $S$.  
This leads to a coupled system of ODEs, with the error controlled by the Brascamp--Lieb inequality. 

As in \cite{de2026sparsifying}, the shifts can be removed if one allows the centered approximating process to be indexed by a new set rather than by a subset of $T$; see \Cref{cor:centered-sparsification}. 
Our approach also yields $L^p$-error sparsifiers for $p \geq 2$; see \Cref{cor:lp}. 
Finally, we show that the dependence on $\eps$ in \Cref{thm:GP-sparsification} is optimal up to constants in the exponent, even if the shifts in the approximating process are allowed to be arbitrary real numbers. 

\begin{restatable}{theorem}{widthonegplbthm}
\label{thm:width-one-gp-lb}
	There is a universal constant $\eps_0>0$ such that for every
	$0<\eps<\eps_0$ there is an integer $n=n(\eps)$ and a bounded set
	$T_\eps\sse\R^n$ with $w(T_\eps)=1$ such that the following holds.  If
	$S\sse T_\eps$ is finite and $\cbra{c_s}_{s\in S}$ are real numbers satisfying
	\[
		\left\|
			\sup_{t\in T_\eps} \bg \cdot t
			-
			\max_{s\in S}\cbra{\bg\cdot s +c_s}
		\right\|_{L^2}
		\leq \eps\,,
	\]
	then $|S|\geq \exp\pbra{\Omega(1/\eps^2)}$.
\end{restatable}

The same example also gives a matching lower bound for $L^1$-error approximators; see \Cref{rem:l1-sphere-lb}.  
Since the $L^2$ guarantee in \Cref{thm:GP-sparsification} is stronger than the corresponding $L^1$ guarantee, this yields the optimal $L^1$ sparsifier size $\exp(\Theta(1/\eps^2))$.  
Thus we close the quantitative gap left by \cite{de2026sparsifying}, between their $\exp(\widetilde\Omega(1/\sqrt\eps))$ lower bound and their double-exponential upper bound. 

\subsection{Applications of \Cref{thm:GP-sparsification}}
\label{subsec:apps}

We next record several geometric consequences of \Cref{thm:GP-sparsification}, including low-dimensional approximation of norms and polyhedral approximation of convex sets under the Gaussian measure. 
These results sharpen analogous consequences of \cite[Theorem~1]{de2026sparsifying}, improving their quantitative bounds by an exponential factor. 

\subsubsection{A Junta Theorem for Norms} 

We begin with a ``junta theorem'' for norms over Gaussian space.
As a motivating example, borrowed from \cite{de2026sparsifying}, consider the Euclidean norm $\psi(x)=\|x\|_2$ on $\R^n$. 
Standard Gaussian concentration bounds imply that $\psi$ can be multiplicatively approximated, within a $(1\pm\eps)$ factor on a $(1-\eps)$ fraction of Gaussian space, by a function of the form
\[
	\phi(x) := C_\eps \sqrt{n}\cdot \sup_{1\leq i \leq 2^{\Theta(1/\eps^2)}} |x_i|\,,
\]
for an appropriate normalization constant $C_\eps \asymp \eps$. 
Thus $\|x\|_2$ can be approximated over Gaussian space by a norm depending on only $2^{\Theta(1/\eps^2)}$ coordinates. 

Motivated by this example and by the notion of a ``junta'' in the analysis of Boolean functions (see, for example, \cite{DMN21} and the references therein), we call a function $\phi:\R^n\to\R$ a \emph{$k$-junta norm} if there is a subspace $E\sse\R^n$ with $\dim(E)\leq k$ and a norm $\nu:E\to\R$ such that 
\[
	\phi(x) = \nu(\pi_E x)\,,
\]
where $\pi_E : \R^n \to E$ is the orthogonal projection onto $E$. 

De, Nadimpalli, O'Donnell, and Servedio~\cite{de2026sparsifying} showed that this junta phenomenon extends, at least qualitatively, to \emph{every} norm on $\R^n$. 
More precisely, \cite[Theorem~2]{de2026sparsifying} proves that every norm on $\R^n$ admits a multiplicative $(1\pm\eps)$-approximation by a $k$-junta norm over Gaussian space, with $k \leq \exp\exp\pbra{O(1/\eps^3)}$. 
As a consequence of \Cref{thm:GP-sparsification}, we improve this dependence by an exponential factor, bringing the general theorem quantitatively closer to the $\ell_2$ example above. 

\begin{restatable}{theorem}{normjuntathm}
\label{thm:norm-junta}
	Let $\eps \in (0, 1/2)$ and suppose $\psi : \R^n \to \R$ is a norm. 
	Then there exists a $k$-junta norm $\phi: \R^n \to \R$ with 
	\[
		k\leq
		\exp\pbra{
			O\pbra{
				\frac{\log(1/\eps)}{\eps^4}
			}
		}
	\]	
	such that 
	\[
		\Prx_{\bg\sim N(0,I_n)}
		\sbra{
			1-\eps
			\leq
			\frac{\psi(\bg)}{\phi(\bg)}
			\leq
			1+\eps
		}
		\geq
		1-\eps\,.
	\]
\end{restatable}

A direct substitution of \Cref{thm:GP-sparsification} into the argument of \cite[Theorem~2]{de2026sparsifying} gives $k \leq \exp(O(1/\eps^5))$. 
To obtain the sharper bound above, we use the $L^p$-sparsifier from \Cref{cor:lp}, which saves a factor of $\eps^{-1}$ in the exponent up to logarithmic factors. 
 
\subsubsection{Polyhedral Approximation of Convex Sets} 
\label{subsec:intro-apps-polytope}

For measurable sets $K,L\sse\R^n$, we write 
\[
	\dG(K,L) := \Prx_{\bg\sim N(0,I_n)}\sbra{\bg \in K\,\triangle\,L}
\]
for the \emph{Gaussian distance} between $K$ and $L$, where 
$K\,\triangle\,L = (K\setminus L)\cup(L\setminus K)$ denotes their symmetric difference. 
This distance is widely used in probability theory, statistics, and theoretical computer science, particularly in derandomization, learning theory, and property testing; see, for example, \cite{DNS23-polytope} and the references therein.

One of the main applications of \cite{de2026sparsifying} was a polyhedral approximation theorem for convex sets that are intersections of low-width halfspaces. 
Formally, we say that a convex set $K\sse\R^n$ has \emph{geometric width} at most $r$ if it can be written as
\[
	K
	=
	\bigcap_{t\in T}
	\cbra{x\in\R^n:t\cdot x\leq r_t},
	\qquad
	T\sse \S^{n-1},
	\qquad
	\abs{r_t}\leq r\,.
\]

De, Nadimpalli, O'Donnell, and Servedio~\cite{de2026sparsifying} showed that every convex set of geometric width $r$ can be approximated, in Gaussian distance, by an intersection of $\exp\exp(O(r^4))$ halfspaces. 
This immediately yields algorithms for agnostic learning and property testing of convex sets of low geometric width. 
Using the $L^p$-sparsifier of \Cref{cor:lp}, derived from \Cref{thm:GP-sparsification}, we obtain the following exponential improvement of \cite[Theorem~3]{de2026sparsifying}. 

\begin{restatable}{theorem}{polytopeapproxthm}
	\label{thm:polytope-approximation}
	Let $r\geq1$ and $0<\eps<1/2$.  Suppose $K\sse\R^n$ has geometric width at
	most $r$.  Then there is a set $L\sse\R^n$, which is an intersection of at
	most
	\[
		\newalpha 
		:=
		\exp\pbra{
			O\pbra{
				\frac{\log(1/\eps)(r+\sqrt{\log(1/\eps)})^4}{\eps^2}
			}
		}
	\]
	halfspaces, such that $\dG(K,L)\leq\eps$.  Equivalently,
	$
		\newalpha 
		=
		\exp(\wt{O}(r^4/\eps^2))
	$.
\end{restatable}

As in \cite{de2026sparsifying}, this approximation theorem immediately yields improved algorithms for agnostic learning and property testing of convex sets of low geometric width; see \Cref{subsec:apps-polytope} for details. 

We also prove a complementary lower bound showing that the quantitative dependence on the geometric width parameter cannot be substantially improved: there are convex sets of geometric width $r$ for which constant-error approximation requires $\exp(\Omega(r^2))$ halfspaces. 
For comparison, Theorem~63 of \cite{DNS23-polytope} already gives an $\exp(\Omega(r))$ lower bound: it shows that $\exp(\Omega(\sqrt n))$ halfspaces are necessary to approximate the Euclidean ball of Gaussian measure $1/2$ to constant error in Gaussian distance, and this ball has geometric width $\Theta(\sqrt n)$. 
We improve this lower bound by considering the $\ell_\infty$ ball of Gaussian measure $1/2$. 

\begin{restatable}{theorem}{cubepolytopelb}
\label{thm:gaussian-cube-polytope-lb} 
	Let $r_n>0$ be chosen so that $C_n:=[-r_n,r_n]^n\sse\R^n$ has Gaussian measure $1/2$. 
	If $P\sse\R^n$ is an intersection of $m$ halfspaces, then
	\[
		\dG(P,C_n)
		\geq
		\frac{1}{2}-\sqrt{\frac{m\ln 2}{2n}}\,.
	\] 
	Consequently, if $\dG(P, C_n) \leq \eps<1/2$, then
	\[
		m\geq \frac{2(1/2-\eps)^2}{\ln 2}\,n\,.
	\]
\end{restatable}

To see the dependence on geometric width, note that $C_n$ has geometric width $r_n$ since it is the intersection of the coordinate halfspaces $\{x_i\leq r_n\}$ and $\{-x_i\leq r_n\}$. 
Moreover, the condition $\gamma_n(C_n)=1/2$ implies $r_n^2=\Theta(\log n)$. 
Thus, for constant $\eps<1/2$, \Cref{thm:gaussian-cube-polytope-lb} shows that approximating $C_n$ requires $\Omega(n)=\exp(\Omega(r_n^2))$ halfspaces. 

\subsection{Technical Overview}
\label{subsec:technical-overview}

We now give a technical overview of our main result, \Cref{thm:GP-sparsification}. 
For simplicity, we assume throughout that the index set $T\sse \R^n$ is finite, which allows us to write $\max$ rather than $\sup$ over $T$. 
We also normalize $w(T) = 1$. 
This is without loss of generality: if $w(T)>0$ then we may rescale $T$, while if it is zero the problem is trivial (see~\Cref{subsec:prelims-GP}).  
 
Recall that the goal is to approximate the maximum of many linear functions by the maximum of far fewer affine functions under the Gaussian distribution. 
More precisely, with the normalization $w(T)=1$, we wish to find a subset $S\sse T$ and shifts $\cbra{c_s\geq 0}_{s\in S}$ such that 
\[
	\vabs{
		\max_{t\in T}\bg\cdot t - \max_{s\in S} \cbra{\bg\cdot s + c_s}
	}_{L^2} \leq \eps\,,
\]
with $|S| = O_\eps(1)$, where the implicit constant may depend on $\eps$ but not on the ambient dimension $n$. 

We view the construction as a clustering problem for the index set $T$.
Namely, we choose a small set of representatives $S\sse T$ and a map $\Pi:T\to S$ assigning each $t\in T$ to a representative $\Pi(t)\in S$, with the convention that $\Pi(s) = s$ for $s \in S$. 
For $s\in S$, let  
\[
	P_s:=\cbra{t\in T:\Pi(t)=s}
\]
be the cluster represented by $s$. 
The starting point, both for our approach and for that of \cite{de2026sparsifying}, is the simple decomposition 
\begin{equation} \label{eq:intro-simple-decomp}
	\max_{t\in T}\bg\cdot t
	=
	\max_{s\in S}
	\left\{
		\bg\cdot s
		+
		\max_{t\in P_s}\bg\cdot (t-s)
	\right\}.
\end{equation}
Thus, after a clustering has been fixed, the original Gaussian process splits into two pieces: a coarse process indexed by the representative elements $s\in S$ and a residual process inside each cluster $P_s$. 
The role of the shift $c_s$ is to compensate for the residual contribution lost when the cluster $P_s$ is collapsed to its representative $s$.  
This raises two questions:  
\begin{itemize}
	\item How should we choose the representatives $S\sse T$ and the clustering map $\Pi:T\to S$?
	\item Once a clustering has been chosen, how should we choose the constants $c_s$ so that replacing each residual process on $P_s \sse T$ by $c_s$ creates only a small error?
\end{itemize}
As we explain below, these two questions are tightly linked: the method used to control the residual error determines what kind of clustering is feasible. 
This is the main point of difference between our approach and that of \cite{de2026sparsifying}. 

\subsubsection{The \cite{de2026sparsifying} Approach} 
\label{subsubsec:DNOS-approach}

We first recall that \cite{de2026sparsifying} obtain an $L^1$-error guarantee, rather than the $L^2$-error guarantee of \Cref{thm:GP-sparsification}. 
Suppose, for now, that a clustering $S \sse T$ and a map $\Pi: T \to S$ are already chosen. 
For a fixed cluster $P_s$, a natural way to collapse the residual process $\max_{t\in P_s}\bg\cdot(t-s)$ to a constant is to replace it by its mean. 
That is, one sets
\begin{equation} \label{eq:DNOS-shift-choices}
	c_s
	:=
	\Ex_{\bg\sim N(0,I_n)}
	\sbra{
		\max_{t\in P_s}\bg\cdot(t-s)
	},
\end{equation}
giving the candidate approximator $\max_{s\in S}\cbra{\bg\cdot s+c_s}$. 
Note that the map $g \mapsto \max_{t\in P_s}g\cdot(t-s)$
is $r_s$-Lipschitz where $r_s:=\sup_{t\in P_s}\|t-s\|_2$ is the radius of the cluster $P_s$. 
As a result, the error within the cluster $P_s$ is readily controlled: standard subgaussian concentration of suprema of Gaussian processes
(see, for example, Appendix~A.5 of \cite{chatterjee2014superconcentration}) gives 
\begin{equation} 
\label{eq:DNOS-single-cluster-error}
	\Ex_{\bg\sim N(0,I_n)}
	\sbra{
		\abs{
			\max_{t\in P_s}\bg\cdot(t-s)-c_s
		}
	}
	\lesssim r_s\,.
\end{equation}
Turning to the total error between this candidate approximator and the original process, 
\Cref{eq:intro-simple-decomp} and the elementary inequality $\abs{\max_s a_s-\max_s b_s}\leq \max_s\abs{a_s-b_s}$ imply   
\begin{equation} \label{eq:DNOS-error} 
	\Ex_{\bg\sim N(0,I_n)}
	\sbra{
		\abs{
			\max_{t\in T}\bg\cdot t
			-
			\max_{s\in S}\cbra{\bg\cdot s+c_s}
		}
	}
	\leq
	\Ex_{\bg\sim N(0,I_n)}
	\sbra{
		\max_{s\in S}
		\abs{
			\max_{t\in P_s}\bg\cdot(t-s)
			-
			c_s
		}
	}.
\end{equation}

This bound reveals why this error-control strategy is incompatible with a naive \emph{single-scale} clustering, such as taking $S$ to be an $\eta$-net. 
Indeed, suppose that $S$ is an $\eta$-net of $T$, and that each $t\in T$ is assigned to a representative $s \in S$ with $\|t-s\|_2 \leq \eta$.
Then $r_s \leq \eta$ for $s \in S$, and so the residual error in each cluster is $O(\eta)$ (cf.~\Cref{eq:DNOS-single-cluster-error}).  
Taking the maximum over $|S|$ such errors costs a factor of order $\sqrt{\log |S|}$ (using subgaussianity of Gaussian suprema once again), yielding  
\[
	\text{R.H.S. of \Cref{eq:DNOS-error}}
	\lesssim
	\eta\sqrt{\log |S|}\,.
\]
However, there exist sets $T \sse \R^n$ satisfying $w(T) = 1$ such that the size of any $\eta$-net of $T$ is $\exp(\Theta(\eta^{-2}))$.  
For such sets, the preceding bound becomes
\[
	\eta\sqrt{\log |S|}
	\lesssim 
	\eta\sqrt{1/\eta^2}
	=
	O(1)
\]
rather than the desired $O(\eps)$. 
\Cref{eg:intro-example} gives a concrete example illustrating this barrier, and also points toward the key change in our proof: with a different choice of centering for the residual processes, a single-scale $\eps$-net will in fact suffice. 

The argument of \cite{de2026sparsifying} overcomes this obstacle by using a \emph{multi-scale} partition of $T$ coming from Talagrand's majorizing measures theorem~\cite[Theorem~2.10.1]{talagrand2022upper}. 
Roughly speaking, their construction stops different branches of the \emph{generic chaining} at different scales. 
At each scale, the diameter of a part controls the subgaussian tail available for the corresponding residual error, while the number of parts determines the cost of the union bound. 
The majorizing measures theorem balances these two effects across scales, making it possible to control all residual errors simultaneously. 
This yields a sparsifier with dimension-independent size, but with double-exponential dependence on the inverse accuracy parameter $\eps^{-1}$. 

\subsubsection{A Guiding Example}
\label{subsubsec:intro-example}

The above discussion suggests that a single-scale net is incompatible with the error analysis used by \cite{de2026sparsifying}. 
The following toy example makes this concrete, and more importantly, identifies the basic issue: the shifts in \Cref{eq:DNOS-shift-choices} center each residual process under its unconditional Gaussian law, rather than under the distribution relevant to the outer maximum. 

\begin{example} \label{eg:intro-example}
	Let $\{e_1, \dots, e_n\} \sse \R^n$ be the standard basis vectors. Consider the index set 
	\[
		T := \cbra{\lambda e_i, 2\lambda e_i : 1 \leq i \leq n} 
		\quad\text{where}\quad 
		\lambda > 0~\text{is chosen so that}~\Ex_{\bg\sim N(0,I_n)}\sbra{\max_{t\in T} \bg\cdot t} = 1\,.
	\]
	Writing $\bM := \max_i \bg_i$ and $x_+ := \max\{0, x\}$, note that 
	\begin{equation} \label{eq:intro-eg-true-process}
		\max_{t\in T} \bg\cdot t = \max_{1\leq i\leq n} \max \{\lambda \bg_i, 2\lambda \bg_i\} = \lambda(\bM + \bM_+)\,,
	\end{equation}
	and hence $\lambda^{-1} = \Ex\sbra{\bM + \bM_{+}} \asymp \sqrt{\log n}$~\cite[Appendix~A.2]{chatterjee2014superconcentration}. 
	Note that the set $S := \{\lambda e_i : 1 \leq i \leq n\}$ forms a $\lambda$-net for $T$ with clusters $P_i := \{\lambda e_i, 2\lambda e_i\}$. 
	The residual process obtained by collapsing the cluster $P_i$ to the representative $\lambda e_i$ is $\max\{0, \lambda \bg'\}$ where $\bg' \sim N(0, 1)$ is a univariate Gaussian. 
	The shifts chosen by  \cite{de2026sparsifying} (cf.~\Cref{eq:DNOS-shift-choices}) are therefore 
	\[
		c_i 
		= \Ex_{\bg' \sim N(0, 1)}\sbra{\lambda\bg'_+} 
		= \frac{\lambda}{\sqrt{2\pi}}\,,
	\]
	giving the candidate approximator  
	\[
		\max_{1 \leq i \leq n} \cbra{\lambda\bg_i + c_i} = \lambda\pbra{\bM + \frac{1}{\sqrt{2\pi}}}\,.
	\]
	It follows, using~\Cref{eq:intro-eg-true-process}, that the $L^2$-error between this approximator and the true process is 
	\[
		\vabs{
			\lambda\pbra{\bM + \bM_+} - 	\lambda\pbra{\bM + \frac{1}{\sqrt{2\pi}}}
		}_{L^2}
		= 
		\lambda\cdot
		\vabs{
			\bM_+ - \frac{1}{\sqrt{2\pi}} 
		}_{L^2}
		=
		\Theta(1)\,,
	\]
	where we used $\Ex\sbra{\bM_+} \asymp \sqrt{\log n}$ and
	$\lambda^{-1}\asymp \sqrt{\log n}$. 
\end{example}

Crucially, the failure in the example above is \emph{not} caused by taking $S$ to be a $\lambda$-net of $T$, but rather by the choice of shifts. 
For this instance, \Cref{eq:intro-eg-true-process} suggests the better choice
\[
	 c_i' := \lambda\cdot\Ex\sbra{\bM_+}\,,
\]
which gives the approximator
\[
	\max_i\cbra{\lambda \bg_i+c_i'}
	=
	\lambda\pbra{\bM+\Ex\sbra{\bM_+}}\,.
\]
The $L^2$-error between this approximator and the true process is
\[
	\vabs{\lambda(\bM+\bM_+) - \lambda\pbra{\bM+\Ex\sbra{\bM_+}}}_{L^2}
	=
	\lambda\cdot\Var\sbra{\bM_+}^{1/2}
	=
	O\pbra{\frac{1}{\log n}}\,,
\]
where we used the standard estimate $\Var\sbra{\bM_+}\asymp \log^{-1} n$~\cite[Appendix~A.2]{chatterjee2014superconcentration} together with $\lambda^{-1}\asymp\sqrt{\log n}$.

This points to the main change in our proof: the residual process should be centered according to its active behavior, not its unconditional law. 
In the example above, the residual from $P_i$ matters on the active region where the representative $\lambda e_i$ wins the outer maximum; on this region, it has the law of $\lambda\bM_+$, not the law of $\lambda\bg'_+$ for an independent one-dimensional Gaussian $\bg'$. 
We therefore center each residual according to its \emph{active law}: the law it sees on the region of Gaussian space where the corresponding representative can influence the maximum; see \Cref{fig:hexagon-shift} for an illustration. 
The difficulty is that these active regions themselves depend on the shifts, making it hard to choose the shifts in ``one shot.''  
We resolve this by choosing the shifts continuously along an interpolation, so that at each infinitesimal step the relevant active regions are already determined. 

\subsubsection{Our Approach} 
\label{subsubsec:our-approach}

\input{sections/fig}

We now turn to our approach. 
The discussion here is heuristic: we work directly with maxima and suppress differentiability issues. 
In the formal argument in \Cref{sec:GP-sparsification}, these issues are handled by replacing the maximum with a softmax. 

By Sudakov's minoration (\Cref{lem:sudakov}) and the normalization $w(T)=1$, there is an $\eps$-net $S\sse T$ of size $|S|\leq \exp(O(1/\eps^2))$. 
Let $\Pi:T\to S$ assign each point of $T$ to its representative in $S$, with the convention that $\Pi(s)=s$ for $s\in S$. 
The technical heart of the proof is to construct shifts for this fixed net so that the resulting approximation has $L^2$-error at most $\eps$.
We construct these shifts dynamically:
as $\theta$ runs from $0$ to $1$, each point $t\in T$ moves linearly toward its representative $\Pi(t)$, and we control the maximum along this path. 
For $t\in T$, define
\[
    v_t\pbra{\theta}
    :=
    \pbra{1-\theta}t+\theta\Pi\pbra{t}\,,
    \qquad
    0\leq \theta\leq 1\,.
\]
Thus $v_t\pbra{0}=t$ and $v_t\pbra{1}=\Pi\pbra{t}$. 
We construct shifts $a_t\pbra{\theta}$ along this interpolation with the initial condition $a_t\pbra{0}=0$. 
With these shifts still to be determined, let the
maximum at time $\theta$ be
\[
    M_\theta\pbra{g}
    :=
    \max_{t\in T}
    \cbra{
        g\cdot v_t\pbra{\theta}+a_t\pbra{\theta}
    }\,.
\]
In particular, 
\[
	M_0(g)=\max_{t\in T}g\cdot t
	\qquad\text{and}\qquad 
	M_1(g)
	=
	\max_{t\in T}
	\left\{
		g\cdot \Pi(t)+a_t(1)
	\right\}
	=:
	\max_{s\in S}
	\left\{
		g\cdot s+c_s
	\right\}\,,
\]
where $c_s:=\max_{t \in P_s} a_t(1)$. 
Because $\Pi(s)=s$, the path for $s$ itself is constant and $a_s(\theta)=0$, so $c_s\geq 0$. 
We will obtain the desired approximator by controlling the motion of
$M_\theta$ in $L^2$ as $\theta$ runs from $0$ to $1$.
Note that the $L^2$-error is 
\begin{align}
    \norm{M_1-M_0}_{L^2}
    &=
    \norm{
        \int_0^1
        \frac{d}{d\theta}M_\theta
        \,d\theta
    }_{L^2} \nonumber \\ 
    &\leq
    \int_0^1
    \norm{
        \frac{d}{d\theta}M_\theta
    }_{L^2}
    \,d\theta\,. \label{eq:overview-our-goal}
\end{align} 
We will thus choose shifts $a_t\pbra{\theta}$ to ensure that the derivative of $M_\theta$ is small in $L^2$. 

For fixed time $\theta$, let $A_t(\theta)\sse\R^n$ be the region on which the affine function corresponding to $t$ is active, ignoring ties and boundary overlap in this heuristic discussion. 
In other words,  
\begin{align*}
	A_t(\theta)
	&:= \cbra{g \in \R^n: g\cdot v_t(\theta)+a_t(\theta) = M_\theta(g)} \\
	&= \bigcap_{t' \in T} \cbra{g\in\R^n : g\cdot\pbra{v_t(\theta) - v_{t'}(\theta)} + \pbra{a_t(\theta) - a_{t'}(\theta)} \geq 0}\,.
\end{align*}
In particular, each active region $A_t\pbra{\theta}$ is an intersection of halfspaces, and hence is convex. 
Additionally, the sets $A_t\pbra{\theta}$ form a partition of $\R^n$ (up to boundary overlap on sets of measure $0$); see \Cref{fig:hexagon-shift} for an illustration. 

Continuing heuristically, on the region $A_t\pbra{\theta}$, the derivative of
the winning affine function is
\[
    \frac{d}{d\theta}M_\theta\pbra{g}
    =
    \frac{d}{d\theta}
    \pbra{
        g\cdot v_t\pbra{\theta}+a_t\pbra{\theta}
    }
    =
    -g\cdot\pbra{t-\Pi\pbra{t}}+a_t'\pbra{\theta}\,.
\]
The interpolation therefore produces the residual term $-g\cdot\pbra{t-\Pi\pbra{t}}$ on the region where $t$ is active. 

We choose the derivative $a_t'\pbra{\theta}$ to cancel the mean of this residual on the active region. Namely, when
$\Prx\sbra{\bg\in A_t\pbra{\theta}}>0$, set
\begin{equation} \label{eq:overview-ODE}
    a_t'\pbra{\theta}
    :=
    \Ex_{\bg\sim N\pbra{0,I_n}}
    \sbra{
        \bg\cdot\pbra{t-\Pi\pbra{t}}
        \,\middle|\,
        \bg\in A_t\pbra{\theta}
    }\,.
\end{equation}
Thus, together with the initial condition $a_t(0) = 0$, \Cref{eq:overview-ODE} specifies a coupled initial-value problem for the shifts $\{a_t(\theta)\}_{t\in T}$. 
The coupling comes from the fact that the active region $A_t(\theta)$ depends on the entire vector of shifts $(a_s(\theta))_{s\in T}$. 
In this heuristic overview, we ignore the (standard) existence and regularity issues and define the shifts $\{a_t(\theta)\}$ by solving this ODE system on the interval $[0,1]$. 

Then, on $A_t\pbra{\theta}$,
\[
    \frac{d}{d\theta}M_\theta\pbra{g}
    =
    -\pbra{
        g\cdot\pbra{t-\Pi\pbra{t}}
        -
        a_t'\pbra{\theta}
    }\,,
\]
which has mean zero after conditioning on $\bg\in A_t\pbra{\theta}$.
As a consequence (and ignoring zero measure events), 
\[
    \Ex\sbra{
        \pbra{
            \frac{d}{d\theta}M_\theta\pbra{\bg}
        }^2
    }
    =
    \sum_{t\in T}
    \Prx\sbra{\bg\in A_t\pbra{\theta}}\,
    \Varx\sbra{
        \bg\cdot\pbra{t-\Pi\pbra{t}}
        \,\middle|\,
        \bg\in A_t\pbra{\theta}
    }\,.
\]
Since $A_t\pbra{\theta}$ is convex, the Gaussian measure conditioned on the convex set $A_t(\theta)$ is still $1$-strongly log-concave. 
Thus we may apply the Brascamp--Lieb inequality, which extends the Gaussian Poincar\'{e} inequality to strongly log-concave measures; see \Cref{thm:BL}. 
This gives
\[
    \Varx\sbra{
        \bg\cdot\pbra{t-\Pi\pbra{t}}
        \,\middle|\,
        \bg\in A_t\pbra{\theta}
    }
    \leq
    \norm{t-\Pi\pbra{t}}_2^2
    \leq
    \eps^2
\]
where the final inequality relies on the fact that $S$ is an $\eps$-net of $T$. 
(We note that the Brascamp--Lieb inequality in \Cref{thm:BL} is stated for $C^2$ potentials. 
Our formal proof avoids the non-smoothness of conditioning on hard active regions by replacing the maximum with a softmax, so that \Cref{thm:BL} applies directly.)
Plugging this into the previous display yields
\[
     \norm{
        \frac{d}{d\theta}M_\theta
    }_{L^2}^2
    =
    \Ex\sbra{
        \pbra{
            \frac{d}{d\theta}M_\theta\pbra{\bg}
        }^2
    }
    \leq
    \sum_{t\in T}
    \Prx\sbra{\bg\in A_t\pbra{\theta}}\cdot\eps^2        
    =
    \eps^2\,,
\]
which, together with \Cref{eq:overview-our-goal}, completes the argument. 

The formal proof in \Cref{sec:GP-sparsification} implements this argument using a smooth maximum, which simultaneously resolves ties, differentiability of $M_\theta$, and regularity of the resulting ODE.

\subsection{Related Work}
\label{subsec:related-work}

We refer the reader to \cite[Section~1.4]{de2026sparsifying} for a comparison with classical dimensionality-reduction results such as Gordon's theorem~\cite{gordon1988milman} and the Johnson--Lindenstrauss lemma~\cite{JohnsonLindenstrauss:84}. 
Here we briefly discuss two related directions that are closest to the present work.

\paragraph{Sparsification of Gaussian Processes.} 

The work of De, Nadimpalli, O'Donnell, and Servedio~\cite{de2026sparsifying} is the direct predecessor of the present paper, and was discussed above in detail. 
We mention one earlier result in a similar spirit, due to Klartag and Mendelson~\cite{klartag2005empirical}. 
In particular, \cite[Lemma~2.3]{klartag2005empirical} implies the following one-sided approximation result for Gaussian suprema: for every $k\geq 1$, there is a subset $S\sse T$ with $|S|\leq 4^k$ such that, with probability at least $1-\exp(-\Omega(k))$,
\[
	\sup_{t\in T}\bX_t
	\leq
	\sup_{s\in S}\bX_s + O(w(T))\,.
\]
This result is qualitatively different from the sparsification theorems considered here. 
It gives a one-sided high-probability comparison at the natural scale $w(T)$, whereas \Cref{thm:GP-sparsification} gives a two-sided $L^2$ approximation with arbitrarily small relative error $\eps\cdot w(T)$, at the cost of allowing shifts in the approximating subprocess. 
We note that both \cite{klartag2005empirical,de2026sparsifying} make use of Talagrand's majorizing measures theorem~\cite[Theorem~2.10.1]{talagrand2022upper}. 

\paragraph{Analogy with Sparsification of CNFs and Submodular Functions over $\zo^n$.} 

Our approximation result for convex bodies can be viewed as a Gaussian space analogue of recent sparsification results for Boolean functions, most notably the width-$w$ CNF sparsification theorem of Lovett, Wu, and Zhang~\cite{lovett2021decision}.  
In both settings, large intersections---of Boolean clauses or Gaussian halfspaces---can be sparsified to approximators whose size depends only on the width and error, and not on the dimension or the original number of constraints. 
Moreover, in both settings the dependence on the relevant width parameter is single-exponential, up to the precise exponent.  
Similarly, \Cref{thm:norm-junta} is analogous to junta approximation results for submodular functions over $\zo^n$~\cite{FKV13,FV16}, in keeping with the standard viewpoint that submodular functions are discrete analogues of convex functions. 
Our results thus add to a growing dictionary between Boolean function analysis and high-dimensional convex geometry: several phenomena have essentially matching quantitative forms in the Boolean and convex settings, despite being proved by rather different techniques~\cite{DNS21itcs,DNS22,DNS23-polytope}. 

%% file: sections/fig.tex
\begin{figure}
\centering
\begin{subfigure}[t]{0.31\textwidth}
\centering
\begin{tikzpicture}[
    scale=1.4,
    every node/.style={font=\small},
    point/.style={circle, fill, inner sep=1.25pt},
    hpoint/.style={circle, fill, inner sep=1.25pt}
]
    \path[use as bounding box] (-1.55,-1.45) rectangle (1.55,1.45);

    \draw[gray!45, thin]
        (0:1.05) --
        (60:1.05) --
        (120:1.05) --
        (180:1.05) --
        (240:1.05) --
        (300:1.05) -- cycle;

    \foreach \ang in {60,120,180,240,300} {
        \node[point] at (\ang:1.05) {};
    }

    \node[hpoint, label=right:{$t$}] at (0:1.05) {};
\end{tikzpicture}
\caption{}
\end{subfigure}
\hspace{0.015\textwidth}
\begin{subfigure}[t]{0.31\textwidth}
\centering
\begin{tikzpicture}[
    scale=1.4,
    every node/.style={font=\small},
    cell/.style={line width=0.7pt},
    origin/.style={circle, fill=black, inner sep=0.8pt}
]
    \path[use as bounding box] (-1.55,-1.45) rectangle (1.55,1.45);

    \begin{scope}
        \clip (-1.55,-1.45) rectangle (1.55,1.45);

        \fill[gray!20]
            (0,0) -- (1.55,0.895) -- (1.55,-0.895) -- cycle;

        \foreach \ang in {-30,30,90,150,210,270} {
            \draw[cell] (0,0) -- (\ang:2.1);
        }
    \end{scope}

    \node at (0.92,0) {$A_t$};
\end{tikzpicture}
\caption{}
\end{subfigure}
\hspace{0.015\textwidth}
\begin{subfigure}[t]{0.31\textwidth}
\centering
\begin{tikzpicture}[
    scale=1.4,
    every node/.style={font=\small},
    cell/.style={line width=0.7pt},
    oldcell/.style={dashed, gray!50!black, line width=0.55pt},
    movearrow/.style={-{Latex[length=1.2mm]}, gray!50!black, line width=0.5pt},
    origin/.style={circle, fill=black, inner sep=0.8pt}
]
    \path[use as bounding box] (-1.55,-1.45) rectangle (1.55,1.45);

    \pgfmathsetmacro{\R}{1.55}
    \pgfmathsetmacro{\del}{0.28}
    \pgfmathsetmacro{\xL}{-\del/2}
    \pgfmathsetmacro{\ya}{\del/(2*sqrt(3))}
    \pgfmathsetmacro{\yb}{2*\del/sqrt(3)}
    \pgfmathsetmacro{\yleft}{\R/sqrt(3)}
    \pgfmathsetmacro{\yright}{\R/sqrt(3)+\yb}

    \begin{scope}
        \clip (-1.55,-1.45) rectangle (1.55,1.45);

        \fill[gray!20]
            (\xL,-\ya) --
            (0,-\yb) --
            (\R,-\yright) --
            (\R,\yright) --
            (0,\yb) --
            (\xL,\ya) -- cycle;

        \draw[oldcell] (0,0) -- (30:2.1);
        \draw[oldcell] (0,0) -- (-30:2.1);

        \draw[cell] (\xL,\ya) -- (0,\yb) -- (\R,\yright);
        \draw[cell] (\xL,-\ya) -- (0,-\yb) -- (\R,-\yright);
        \draw[cell] (\xL,-\ya) -- (\xL,\ya);

        \draw[cell] (0,\yb) -- (0,1.45);
        \draw[cell] (\xL,\ya) -- (-\R,\yleft);
        \draw[cell] (\xL,-\ya) -- (-\R,-\yleft);
        \draw[cell] (0,-\yb) -- (0,-1.45);

        \draw[movearrow] (0.82,0.47) -- +(120:0.23);
        \draw[movearrow] (0.82,-0.47) -- +(-120:0.23);
    \end{scope}

    \node at (0.92,0) {$A_t'$};
\end{tikzpicture}
\caption{}
\label{subfig:c}
\end{subfigure}

\caption{
	An illustration of the active region $A_t$. 
	(a)~A finite index set $T\subset \R^2$. 
	(b)~The active region $A_t = \{g \in \R^2 : g\cdot t \geq g\cdot t'~\text{for all}~t'\in T\}$ for the unshifted maximum $\max_{t'\in T} g \cdot t'$. 
	(c)~If only the piece indexed by $t$ is shifted from $g\cdot t$ to $g\cdot t+\delta$ for some $\delta > 0$, the active region becomes $A'_t$. The dashed rays show the original active region $A_t$; note that only the comparison boundaries involving $t$ move. 
}
\label{fig:hexagon-shift}
\end{figure}
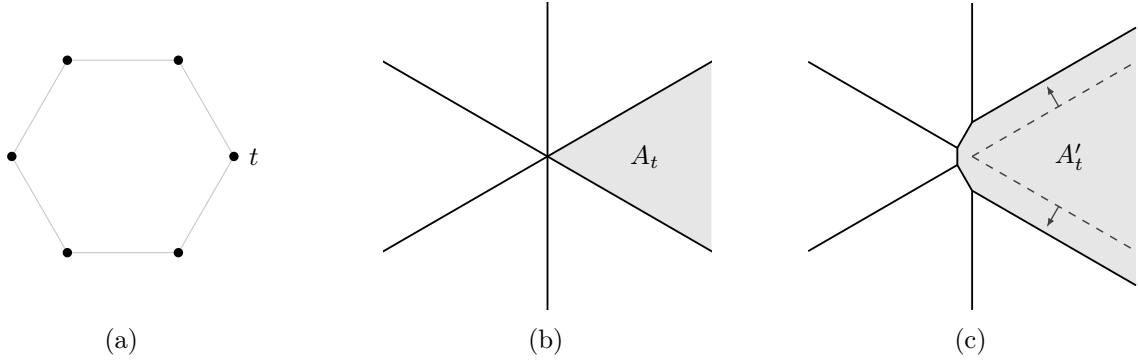

%% file: sections/preliminaries.tex
\section{Preliminaries}
\label{sec:prelims} 

We use boldfaced letters, such as $\bx$ and $\bX$, to denote random variables, which may be real- or vector-valued; the intended type will be clear from context.  
We write $\bx \sim \calD$ to mean that $\bx$ is distributed according to $\calD$. 
Throughout, $\S^{n-1}$ denotes the Euclidean unit sphere in $\R^n$, and $\|\cdot\|_2$ denotes the Euclidean norm. 
For a set $T \sse \R^n$, an $\eps$-net of $T$ is a set $S \sse \R^n$ such that for every $t \in T$, there is an $s \in S$ with $\|t-s\|_2 \leq \eps$. 

\subsection{Gaussian Random Variables}
\label{subsec:prelims-gaussian-rvs}

We write $N(0, I_n)$ for the $n$-dimensional standard Gaussian distribution, and write $\gamma_n(\cdot)$ for the corresponding measure. 
In other words, 
\[
	d\gamma_n(x) = \frac{1}{(2\pi)^{n/2}} e^{-\|x\|^2/2}\,dx\,. 
\] 
We will make use of the polar decomposition of a standard Gaussian: if $\bg\sim N(0, I_n)$, then $\bg = \bR \, \btheta$ where $\bR = \|\bg\|_2 \sim \chi(n)$ is distributed according to the chi distribution with $n$ degrees of freedom and $\btheta = \|\bg\|_2^{-1} \bg \sim \S^{n-1}$ is distributed according to the Haar measure on the unit sphere. 
Additionally, the random variables $\bR$ and $\btheta$ are independent.  

For $1 \leq p < \infty$, we define the $L^p(\gamma_n)$ norm as  
\[
	\|f\|_{L^p(\gamma_n)} := \Ex_{\bg\sim N(0,I_n)}\sbra{\abs{f(\bg)}^p}^{1/p}\,.
\]
When the ambient Gaussian space is clear, we abbreviate this norm by $\|F\|_{L^p}$ and write $L^p=L^p(\gamma_n)$. 
More generally, for a measure $\mu$ on $\R^n$ we define $\|\cdot\|_{L^p(\mu)}$ analogously. 

We recall the following well-known identity for Gaussian random variables~\cite[Appendix~A.3]{chatterjee2014superconcentration}. 

\begin{lemma}[Gaussian integration by parts] \label{lemma:gaussian-IBP}
	Suppose $f: \R \to \R$ is a differentiable function such that $f'\in L^1(\gamma_1) $. 
	Then 
	\[
		\Ex_{\bg\sim N(0,1)}\sbra{\bg\, f(\bg)} = \Ex_{\bg\sim N(0,1)}\sbra{f'(\bg)}\,.
	\]
	Consequently, for $F : \R^n \to \R$ with $F \in C^1$ and $F, \frac{\partial}{\partial x_i} F(x) \in L^1(\gamma_n)$ for $1 \leq i \leq n$, we have   
	\[
		\Ex_{\bg\sim N(0, I_n)}\sbra{\bg_j F(\bg)}
		=
		\Ex_{\bg\sim N(0, I_n)}\sbra{\frac{\partial}{\partial \bg_j} F(\bg)}
	\]
	for each coordinate $1 \leq j \leq n$. 
\end{lemma}

Finally, we will require the following concentration inequality for Lipschitz functions of Gaussian random variables.

\begin{proposition}[Theorem~5.2.2 of~\cite{vershynin2018high}]
\label{prop:lipschitz-concentration}
	Suppose $f:\R^n\to\R$ is an $L$-Lipschitz function. Then 
	\[
		\Prx_{\bg\sim N(0,I_n)}\sbra{\abs{f(\bg) - \Ex[f]} \geq t} \leq 2\exp\pbra{-\frac{t^2}{2L^2}}.
	\]
\end{proposition}

\subsection{Gaussian Processes and Sudakov's Minoration}
\label{subsec:prelims-GP} 

Since all Gaussian processes considered here are canonical processes indexed by subsets of Euclidean space, no separability issues arise: for $\bX_t=\bg\cdot t$ the sample paths are continuous in $t$, and hence suprema over $T\sse\R^n$ may be taken over countable dense subsets of $T$. 
Note that for a set $T \sse \R^n$, we have $w(T) \geq 0$ with  equality if and only if $|T| = 1$. 
The reverse direction is clear; to see the forward direction, note that if $s, t \in T$ with $s \neq t$ then 
\[
	w(T) \geq \Ex_{\bg\sim N(0, I_n)}\sbra{\max\cbra{\bg\cdot s, \bg\cdot t}} = \frac{1}{2}\Ex_{\bg\sim N(0, I_n)}\sbra{\abs{\bg\cdot (s-t)}} > 0\,,
\]
Additionally, $w(T)$ is finite if and only if $T$ is bounded. 

\begin{lemma}[Sudakov's minoration; Theorem~7.4.1 of~\cite{vershynin2018high}] 
\label{lem:sudakov}
	Let $T \sse \R^n$ and suppose $S \sse T$ such that for every $s, s' \in S$, we have $\|s-s'\|_2 \geq \eps$.	
	Then we have 
	\[
		w(T) = \Omega\pbra{\eps \sqrt{\log |S|}}\,. 
	\]
\end{lemma}

Since every maximal $\eps$-separated (in the sense of \Cref{lem:sudakov}) set $S \sse T$ is an $\eps$-net of $T$, it follows from \Cref{lem:sudakov} that $T$ has an $\eps$-net of size 
\begin{equation} \label{eq:sudakov-net-size}
	\exp\pbra{O\pbra{\frac{w(T)^2}{\eps^2}}}\,.
\end{equation}

\subsection{The Brascamp--Lieb Inequality}
\label{subsec:BL}

Our proof of \Cref{thm:GP-sparsification} will crucially rely on the Brascamp--Lieb inequality~\cite{BrascampLieb:76}: 

\begin{lemma}[Brascamp--Lieb variance inequality] \label{thm:BL}
	Let $\mu$ be a probability measure on $\R^n$ with density $d\mu(x) \propto e^{-V(x)}\,dx$ where $V\in C^2$ and $\nabla^2 V(x)\succeq I$ for every $x\in\R^n$. 
	Then, for any differentiable function $f:\R^n\to\R$, we have
	\[
		\Varx_{\bx\sim\mu}\sbra{f(\bx)}
		\leq 
		\Ex_{\bx\sim\mu}\sbra{\|\nabla f(\bx)\|_2^2}\,.
	\]
	In particular, for every $u\in\R^n$,
	\[
		\Varx_{\bx\sim\mu}\sbra{\bx\cdot u}
		\leq 
		\|u\|_2^2\,.
	\]
\end{lemma}

%% file: sections/improved-sparsification.tex
\section{Improved Sparsification of Gaussian Processes}
\label{sec:GP-sparsification}

We now prove \Cref{thm:GP-sparsification} by making the interpolation argument sketched in~\Cref{subsec:technical-overview} rigorous.
The main point is to choose shifts along the interpolation so that the corresponding softmax path moves only a small amount in $L^2(\gamma_n)$. 

\subsection{The Principal Technical Lemma}
\label{subsec:collapse-lemma}

The following lemma shows the existence of good shifts that produce small error for a finite index set. 

\begin{lemma} \label{lem:soft-active-clustering}
	Suppose $T \sse \R^n$ is finite and non-empty and let $\eps > 0$. 
	Additionally suppose $S \sse \R^n$ and the map $\Pi : T \to S$ satisfies $\|t - \Pi(t)\|_2 \leq \eps$ for every $t\in T$. 
	For $t \in T$ and $\theta \in [0,1]$, define 
	\[
		u_t := t - \Pi(t) 
		\qquad\text{and}\qquad 
		v_t(\theta) := t - \theta u_t\,. 
	\]
	Finally, fix $\beta > 0$. 
	Then there exist absolutely continuous functions $a_t : [0,1] \to \R$ with $a_t(0) = 0$ such that the softmax function $H_\theta : \R^n \to \R$ defined as 
	\begin{equation} \label{eq:softmax-def}
		H_\theta(g) := \frac{1}{\beta}\log\pbra{\sum_{t\in T} \exp\pbra{\beta\pbra{g\cdot v_t(\theta)+a_t(\theta)}}}
	\end{equation}
	satisfies $\|H_1-H_0\|_{L^2} \leq \eps$. 
\end{lemma}

\begin{proof}
	We first reduce the goal to bounding the derivative of the softmax interpolation. 
	Suppose for the moment that the shifts $a_t$ have been constructed as absolutely continuous functions with bounded derivatives.
	Then \Cref{lem:l2-chain} implies that the softmax path $\theta\mapsto H_\theta$ is absolutely continuous as an $L^2$-valued curve.
	The same lemma also identifies its $L^2$-derivative with the pointwise derivative for a.e.~$\theta$. 
	Consequently, by the fundamental theorem of calculus, we have 
	\begin{align}
		\|H_1-H_0\|_{L^2}
		&=
		\left\|
			\int_0^1 \frac{d}{d\theta} H_\theta \,d\theta
		\right\|_{L^2}
		\label{eq:main-proof-FTC} \\
		&\leq
		\int_0^1
		\left\|
			\frac{d}{d\theta} H_\theta
		\right\|_{L^2}
		\,d\theta\,.
		\nonumber
	\end{align}
	It therefore suffices to construct absolutely continuous shifts $a_t:[0,1]\to\R$ with bounded derivatives such that 
	\begin{equation}
		\label{eq:softmax-goal}
		\left\|
			\frac{d}{d\theta}H_\theta
		\right\|_{L^2}
		\leq
		\eps
		\qquad
		\text{for a.e.}~\theta\in[0,1]\,.
	\end{equation}
	
	We will obtain the shifts $a_t$ as the solution to an initial-value ODE.
	The ODE is designed so that, at each time $\theta$, the derivative $a_t'(\theta)$ cancels the average residual $\bg\cdot u_t$ under the law in which the affine piece corresponding to the element $t$ is softly active. 
	
	\paragraph{Step 1: Defining the ODE.} 	
	
	To define the vector field, we first freeze time $\theta$ and plug in an arbitrary shift vector $\alpha := (\alpha_t)_{t \in T} \in \R^{|T|}$.
	For this frozen pair $(\theta,\alpha)$, define the pointwise softmax weights 
	\[
		\omega_t^{\theta,\alpha}(g)
		:=
		\frac{
		\exp\left(\beta(g\cdot v_t(\theta)+\alpha_t)\right)
		}{
		\sum_{r\in T}
		\exp\left(\beta(g\cdot v_r(\theta)+\alpha_r)\right)
		}\,.
	\]
	Note that $\omega_t^{\theta,\alpha} \geq 0$ and $\sum_{t \in T} \omega_t^{\theta,\alpha} = 1$. 
	Let 
	\[
		\rho_t(\theta,\alpha) := \Ex_{\bg\sim N(0, I_n)}\sbra{\omega_t^{\theta,\alpha}(\bg)} > 0
		\qquad\text{and}\qquad 
		m_t(\theta,\alpha) := \frac{1}{\rho_t(\theta,\alpha)}\Ex_{\bg\sim N(0, I_n)}\sbra{\omega_t^{\theta,\alpha}(\bg)(\bg\cdot u_t)}\,.
	\]
	Note that $\rho_t(\theta,\alpha)$ is the average soft-active mass of the piece corresponding to $t$, and $m_t(\theta,\alpha)$ is the average residual $\bg\cdot u_t$ conditioned on that piece being (softly) active. 
	Equivalently, if
	\begin{equation} \label{eq:mu-def}
		d\mu_t^{\theta,\alpha}(g)
		:=
		\frac{\omega_t^{\theta,\alpha}(g)}{\rho_t(\theta,\alpha)}
		\,d\gamma_n(g)\,,
		\qquad\text{then}\qquad 
		m_t(\theta,\alpha)
		=
		\Ex_{\bg\sim\mu_t^{\theta,\alpha}}\sbra{\bg\cdot u_t}\,.
	\end{equation}
	
	We now define the actual shift functions $a_t$ through the following system of ordinary differential equations.
	Writing $\bar a(\theta):=(a_t(\theta))_{t\in T}\in\R^{|T|}$ for the vector of shifts at time $\theta$, we set 
	\begin{equation} \label{eq:intercept-ode}
		a_t'(\theta)=m_t(\theta,\bar{a}(\theta))\,,
		\qquad
		a_t(0)=0\,,
		\qquad 
		t \in T\,. 
	\end{equation}
	Note that the derivative with respect to $\theta$ of $g\cdot v_t(\theta)+a_t(\theta)$ is $-g\cdot u_t+a_t'(\theta)$.
	The ODE chooses $a_t'(\theta)=m_t(\theta,\bar{a}(\theta))$, so this derivative has mean zero under \smash{$\mu_t^{\theta,\bar{a}(\theta)}$} (in particular, this is the softmax analogue of the active-centering rule from \Cref{subsubsec:our-approach}).

	\paragraph{Step 2: Global solvability of the shift ODE.}
	
	We now justify that the coupled ODE defining the shifts has a solution on the whole interval $[0,1]$ (we do not require uniqueness). 
	We use the following standard consequence of Peano's existence theorem~\cite[Theorem 2.19]{Tes12}: a continuous bounded vector field $F:[0,1]\times\R^N \to \R^N$ admits a global absolutely continuous solution to every initial-value problem on $[0,1]$. 
	Indeed, Peano gives a local solution, while boundedness of $F$ makes every solution Lipschitz and hence prevents finite-time blowup; this is recorded in the appendix as \Cref{lem:bounded-ode}.
	We check that the vector field 
	\[
		F:[0,1]\times\R^{|T|}\to\R^{|T|},
		\qquad
		F_t(\theta,\alpha)
		:=
		m_t(\theta,\alpha)
	\]
	is continuous and uniformly bounded.
	In particular, \Cref{eq:intercept-ode} is the initial-value problem $\bar{a}'(\theta)=F(\theta,\bar{a}(\theta))$ with initial condition $\bar{a}(0)=0$.
	Here a solution means an absolutely continuous path $\bar{a}:[0,1]\to\R^{|T|}$ whose derivative satisfies the equation for a.e.~$\theta$. 
	
	\begin{itemize}
		\item \textbf{Continuity.}  
		We first show that $F$ is continuous.
		For each fixed $g \in \R^n$, the quantity $\omega_t^{\theta,\alpha}(g)$ is continuous in $(\theta,\alpha)$.
		Moreover,
		\[
			0\leq \omega_t^{\theta,\alpha}(g)\leq 1,
			\qquad
			\abs{\omega_t^{\theta,\alpha}(g)(g\cdot u_t)}
			\leq
			\|u_t\|_2\cdot\|g\|_2
			\leq
			\eps\|g\|_2\,.
		\]
		Since $\|g\|_2$ is integrable under $\gamma_n$, dominated convergence shows that both
		\[
			(\theta,\alpha)
			\mapsto
			\Ex_{\bg\sim N(0,I_n)}\sbra{\omega_t^{\theta,\alpha}(\bg)}
			\qquad\text{and}\qquad 
			(\theta,\alpha)
			\mapsto
			\Ex_{\bg\sim N(0,I_n)}\sbra{\omega_t^{\theta,\alpha}(\bg)(\bg\cdot u_t)}
		\]
		are continuous.
		Also $\rho_t(\theta,\alpha)>0$, since $\omega_t^{\theta,\alpha}(g)>0$ pointwise.
		Hence $m_t(\theta,\alpha)$, and therefore $F_t(\theta,\alpha)$, is continuous.

		\item \textbf{Boundedness.} 
		We next prove a bound on $F$ that is uniform over both $\theta$ and $\alpha$.
		Fix $(\theta,\alpha)$ and $t\in T$.
		Recall that by the definition of $\mu_t^{\theta,\alpha}$ (\Cref{eq:mu-def}),
		\[
			m_t(\theta,\alpha)
			=
			\Ex_{\bg\sim\mu_t^{\theta,\alpha}}\sbra{\bg\cdot u_t}\,.
		\]
		Thanks to Gaussian integration by parts (\Cref{lemma:gaussian-IBP}), applied to the $j^\text{th}$ coordinate after conditioning on the other coordinates, we have 
		\begin{align*}
			\Ex_{\bg\sim\mu_t^{\theta,\alpha}}\sbra{\bg_j}
			&=
			\frac{1}{\rho_t(\theta,\alpha)}
			\Ex_{\bg\sim N(0,I_n)}\sbra{\bg_j \omega_t^{\theta,\alpha}(\bg)}
			\\
			&=
			\frac{1}{\rho_t(\theta,\alpha)}
			\Ex_{\bg\sim N(0,I_n)}\sbra{\frac{\partial}{\partial \bg_j} \omega_t^{\theta,\alpha}(\bg)}
			\\
			&=
			\Ex_{\bg\sim\mu_t^{\theta,\alpha}}\sbra{\frac{\partial}{\partial \bg_j}\log \omega_t^{\theta,\alpha}(\bg)}\,.
		\end{align*}
		In particular, we have 
		\smash{$
			\Ex_{\bg\sim\mu_t^{\theta,\alpha}}\sbra{\bg}
			=
			\Ex_{\bg\sim\mu_t^{\theta,\alpha}}\sbra{\nabla_g\log \omega_t^{\theta,\alpha}(\bg)}
		$}.
		A direct differentiation of the softmax weights gives
		\[
			\nabla_g\log \omega_t^{\theta,\alpha}(g)
			=
			\beta\pbra{
			v_t(\theta)
			-
			\sum_{r\in T}\omega_r^{\theta,\alpha}(g)v_r(\theta)
			}\,.
		\]
		The vector $\sum_{r\in T}\omega_r^{\theta,\alpha}(g)v_r(\theta)$ is a convex combination of the points $\cbra{v_r(\theta):r\in T}$.
		Write
		\[
			D_\theta
			:=
			\diam\cbra{v_r(\theta):r\in T},
			\qquad
			D
			:=
			\sup_{0\leq\theta\leq 1}D_\theta\,.
		\]
		Since $T$ is finite and the paths $\theta\mapsto v_r(\theta)$ are continuous, $D<\infty$.
		Hence, 
		\[
			\left\|
			\nabla_g\log \omega_t^{\theta,\alpha}(g)
			\right\|_2
			\leq
			\beta D\,,
			\qquad\text{and so}\qquad 
			\bigg\|
			\Ex_{\bg\sim\mu_t^{\theta,\alpha}}\sbra{\bg}
			\bigg\|_2
			\leq
			\beta D\,.
		\]
		Since $m_t(\theta,\alpha)=u_t\cdot\Ex_{\bg\sim\mu_t^{\theta,\alpha}}\sbra{\bg}$, we get
		\[
			\abs{m_t(\theta,\alpha)}
			\leq
			\|u_t\|_2 
			\cdot 
			\left\|
			\Ex_{\bg\sim\mu_t^{\theta,\alpha}}\sbra{\bg}
			\right\|_2
			\leq
			\eps\beta D\,.
		\]
		Consequently,
		\[
			\|F(\theta,\alpha)\|_2
			\leq
			|T|^{1/2}\eps\beta D
			\qquad
			\text{for all }(\theta,\alpha)\in[0,1]\times\R^{|T|}\,.
		\]
	\end{itemize}
	Thus $F$ is continuous and uniformly bounded.
	Applying \Cref{lem:bounded-ode} with $M=|T|^{1/2}\eps\beta D$ gives a global absolutely continuous solution $\bar{a}:[0,1]\to\R^{|T|}$ of
	\[
		\bar{a}'(\theta)=F(\theta,\bar{a}(\theta)),
		\qquad
		\bar{a}(0)=0\,.
	\]
	Moreover,
	\[
		\|\bar{a}'(\theta)\|_2
		\leq
		|T|^{1/2}\eps\beta D
		\qquad
		\text{for a.e. }\theta\in[0,1]\,.
	\]
	In particular, each scalar derivative $a_t'$ is bounded a.e.
	Together with the bounded derivatives of the paths $v_t(\theta)$, this is exactly the regularity needed to apply \Cref{lem:l2-chain} to the softmax path $H_\theta$, and hence to justify the fundamental theorem of calculus step in \Cref{eq:main-proof-FTC}.

	\paragraph{Step 3: Rewriting the centered derivative.} 
	
	The rest of the proof is the softmax analogue of the overview from \Cref{subsec:technical-overview}. 
	The only new feature is that the softmax weights form a mixture rather than a hard partition, so the exact partition identity becomes an (application of Jensen's) inequality. 
	
	We define $H_\theta$ as in~\Cref{eq:softmax-def} using $\bar{a}(\theta) = (a_t(\theta))_{t \in T}$.  
	Fix $\theta$ in the full-measure subset of $[0,1]$ on which \Cref{eq:intercept-ode} holds and on which the derivative formula from \Cref{lem:l2-chain} applies.
	Then
	\begin{equation} \label{eq:H-derivative-raw}
		\frac{d}{d\theta}H_\theta(g)
		=
		\sum_{t\in T}
		\omega_t^{\theta,\bar{a}(\theta)}(g)
		\left(-g\cdot u_t+a_t'(\theta)\right)\,.
	\end{equation}
	In particular, thanks to \Cref{eq:intercept-ode}, we can write the above as 
	\begin{equation} \label{eq:H-derivative-centered}
		\frac{d}{d\theta}H_\theta(g)
		=
		-\sum_{t \in T} \omega_t^{\theta,\bar{a}(\theta)}(g)\pbra{g\cdot u_t - m_t(\theta, \bar{a}(\theta))}\,. 
	\end{equation} 

	To lighten notation, write
	\[
		\omega_t(g):=\omega_t^{\theta,\bar{a}(\theta)}(g)\,,
		\qquad
		\rho_t:=\rho_t(\theta,\bar{a}(\theta))\,,
		\qquad
		m_t:=m_t(\theta,\bar{a}(\theta))\,,
		\qquad 
		\mu_t=\mu_t^{\theta,\bar{a}(\theta)}\,.
	\]
	For each fixed $g$, recall that the weights $\omega_t(g)$ are non-negative and sum to one. 
	Thus, applying Jensen's inequality to the convex function $x\mapsto x^2$, \Cref{eq:H-derivative-centered} gives
	\begin{equation} 
		\pbra{\frac{d}{d\theta}H_\theta(g)}^2
		=
		\pbra{
		\sum_{t\in T}\omega_t(g)\pbra{g\cdot u_t-m_t}
		}^2
		\leq 
		\sum_{t\in T}
		\omega_t(g)\pbra{g\cdot u_t-m_t}^2\,.
		\label{eq:weighted-jensen-softmax}
	\end{equation}
	Integrating \Cref{eq:weighted-jensen-softmax} over $\bg\sim N(0,I_n)$ gives
	\begin{align}
		\Ex_{\bg\sim N(0,I_n)}
		\sbra{
		\pbra{\frac{d}{d\theta}H_\theta(\bg)}^2
		}
		&\leq
		\sum_{t\in T}
		\Ex_{\bg\sim N(0,I_n)}
		\sbra{
		\omega_t(\bg)\pbra{\bg\cdot u_t-m_t}^2
		}\,.
		\label{eq:jensen-soft-active}
	\end{align}
	By the definition of $\mu_t$ in \Cref{eq:mu-def}, specialized to our fixed $\theta$ and $\bar a(\theta)$,
	\[
		d\mu_t(g)
		=
		\frac{\omega_t(g)}{\rho_t}\,d\gamma_n(g)\,.
	\]
	Hence
	\[
		\Ex_{\bg\sim N(0,I_n)}
		\sbra{
		\omega_t(\bg)\pbra{\bg\cdot u_t-m_t}^2
		}
		=
		\rho_t
		\cdot 
		\Ex_{\bg\sim\mu_t}
		\sbra{
		\pbra{\bg\cdot u_t-m_t}^2
		}\,.
	\]
	Moreover, recall that $m_t = \Ex_{\bg\sim\mu_t}\sbra{\bg\cdot u_t}$.
	Therefore, 
	\begin{align}
		\sum_{t\in T}
		\Ex_{\bg\sim N(0,I_n)}
		\sbra{
		\omega_t(\bg)\pbra{\bg\cdot u_t-m_t}^2
		}
		&=
		\sum_{t\in T}
		\rho_t 
		\cdot 
		\Varx_{\bg\sim\mu_t}\sbra{\bg\cdot u_t}\,.
		\label{eq:variance-mixture}
	\end{align}
	
	\paragraph{Step 4: Controlling the derivative using Brascamp--Lieb.} 

	It remains to bound the variances in \Cref{eq:variance-mixture}.
	Fix $t\in T$ and recall that $d\mu_t(g) \propto \omega_t(g)\, d\gamma_n(g)$. 
	Therefore, up to an additive constant, the negative log-density of $\mu_t$ is
	\[
		V_t(g)
		=
		\frac12\|g\|_2^2
		-
		\beta\pbra{g\cdot v_t(\theta)+a_t(\theta)}
		+
		\log\pbra{
		\sum_{t'\in T}
		\exp\pbra{\beta\pbra{g\cdot v_{t'}(\theta)+a_{t'}(\theta)}}
		}\,.
	\]
	Since the middle term is affine in $g$, and since log-sum-exp is convex, we have
	\[
		\nabla^2 V_t(g)
		=
		I
		+
		\nabla^2
		\log\pbra{
		\sum_{t'\in T}
		\exp\pbra{\beta\pbra{g\cdot v_{t'}(\theta)+a_{t'}(\theta)}}
		}
		\succeq I\,.
	\]
	Applying the Brascamp--Lieb inequality from \Cref{thm:BL}, we obtain
	$
		\Varx_{\bg\sim\mu_t}\sbra{\bg\cdot u_t}
		\leq
		\|u_t\|_2^2
		\leq
		\eps^2
	$.
	Combining this with \Cref{eq:jensen-soft-active,eq:variance-mixture} and $\sum_{t\in T}\rho_t=1$, we get
	\[
		\Ex_{\bg\sim N(0,I_n)}
		\sbra{
		\pbra{\frac{d}{d\theta}H_\theta(\bg)}^2
		}
		\leq 
		\sum_{t\in T}\rho_t \cdot \eps^2
		=
		\eps^2\,.
	\]
	Since this holds for every $\theta$ in the full-measure set fixed above, this establishes \Cref{eq:softmax-goal} and in turn completes the proof of \Cref{lem:soft-active-clustering}. 
\end{proof} 

\subsection{Proof of~\Cref{thm:GP-sparsification}}
\label{subsec:putting-it-all-together} 

We now turn to the proof of our main result: 

\gpsparsificationthm*

\begin{proof}[Proof of~\Cref{thm:GP-sparsification}]
	We first prove the theorem for finite $T \sse \R^n$, from which the result for bounded $T \sse \R^n$ follows via a standard limiting argument. 
	Note that if $w(T)=0$ then the result is trivial (see \Cref{subsec:prelims-GP}). 
	We therefore assume that $w(T)>0$. 
	Set
	\[
		r
		:=
		\frac{\eps}{2}
		\cdot
		w(T)\,.
	\]
	Let $S\sse T$ be a maximal $r$-separated set, and note that $S$ must be an $r$-net of $T$. 
	Applying \Cref{lem:sudakov} to the $r$-separated set $S$ gives
	\[
		w(T)
		\geq
		\Omega\pbra{r\sqrt{\log |S|}}\,,
	\]
	and hence
	\begin{equation} \label{eq:net-size}
		|S|
		\leq
		\exp\pbra{O\pbra{
		\frac{
		w(T)^2
		}{r^2}
		}}
		=
		\exp\pbra{O\pbra{\frac{1}{\eps^2}}}\,.
	\end{equation}
	Assign each $t \in T$ to some $\Pi(t) \in S$ satisfying $\|t - \Pi(t)\|_2 \leq r$. 
	We additionally assume without loss of generality that $\Pi(s) = s$ for every $s \in S$. 
	
	Choose $\beta > 0$ so that $2\log |T|/\beta \leq r$.
	Applying \Cref{lem:soft-active-clustering} to the map $\Pi:T\to S$ with error parameter $r$ and this value of $\beta$ gives shifts $a_t:[0,1]\to\R$ and the corresponding softmax interpolation $H_\theta$.
	Since $a_t(0)=0$ and $v_t(0)=t$, we have
	\[
		H_0(g)
		=
		\frac{1}{\beta}\log\pbra{\sum_{t\in T}e^{\beta(g\cdot t)}}\,,
		\qquad\text{and so}\qquad 
		\max_{t\in T} g\cdot t
		\leq 
		H_0(g)
		\leq 
		\max_{t\in T} g\cdot t+\frac{\log |T|}{\beta}\,.
	\]
	We similarly have 
	\[
		\max_{t\in T}\{g\cdot \Pi(t)+a_t(1)\}
		\leq
		H_1(g)
		\leq
		\max_{t\in T}\{g\cdot \Pi(t)+a_t(1)\}+\frac{\log |T|}{\beta}\,.
	\]
	Since $\|H_1-H_0\|_{L^2}\leq r$ by \Cref{lem:soft-active-clustering}, the triangle inequality gives
	\[
		\Ex_{\bg\sim N(0, I_n)}\sbra{\pbra{
		\max_{t\in T}\bg\cdot t
		-
		\max_{t\in T}\{\bg\cdot \Pi(t)+a_t(1)\}
		}^2}^{1/2}
		\leq 
		r+\frac{2\log |T|}{\beta}\,.
	\]
	Set $b_t=a_t(1)$.
	By our choice of $\beta$, we have
	\begin{equation} \label{eq:ahhh}
		\Ex_{\bg\sim N(0, I_n)}\sbra{\pbra{
		\max_{t\in T}\bg\cdot t
		-
		\max_{t\in T}\{\bg\cdot \Pi(t)+b_t\}
		}^2}^{1/2}
		\leq 
		2r
		=
		\eps\cdot w(T)\,. 
	\end{equation}
	
	For each $s \in S$, we define $c_s=\max\cbra{b_t:t\in T,\ \Pi(t)=s}$ and note that 
	\[
		\max_{t\in T}\cbra{g\cdot \Pi(t)+b_t}
		=
		\max_{s\in S}\cbra{g\cdot s+c_s}\,.
	\]
	These shifts are non-negative.
	Indeed, since $\Pi(s)=s$, we have $u_s=s-\Pi(s)=0$, so $m_s\equiv 0$ and the ODE gives $a_s(\theta)\equiv0$.
	Thus $b_s=0$, and since $s \in \Pi^{-1}(s)$ we have $c_s\geq b_s=0$. 
	Combining this with \Cref{eq:ahhh}, we get
	\[
		\Ex_{\bg\sim N(0, I_n)}
		\sbra{
		\pbra{
		\max_{t\in T} \bg\cdot t
		-
		\max_{s\in S}\cbra{\bg\cdot s + c_s}
		}^{2}
		}^{1/2}
		\leq
		\eps\cdot w(T)\,.
	\]
	This proves the theorem for finite $T$.

	We now pass to a general bounded set $T\sse\R^n$. 
	Since bounded subsets of $\R^n$ are separable, choose a countable dense subset $\{t_1,t_2,\ldots\}\sse T$, and set $T_m:=\{t_1,\ldots,t_m\}$.
	For each fixed $g\in\R^n$, the map $t\mapsto g\cdot t$ is continuous, so
	$
		\max_{t\in T_m}g\cdot t
		\uparrow
		\sup_{t\in T}g\cdot t
	$.
	Since $T$ is bounded, the difference is dominated by a constant multiple of $\|g\|_2$.
	Thus dominated convergence gives
	\[
		\left\|
		\sup_{t\in T}\bg\cdot t-\max_{t\in T_m}\bg\cdot t
		\right\|_{L^2}
		\to 0\,.
	\]
	Choose $m$ large enough so that this $L^2$-norm is at most
	\[
		\frac{\eps}{2}
		\cdot
		w(T)\,.
	\]
	Applying the finite case to $T_m$ with error parameter $\eps/2$ gives a set $S\sse T_m\sse T$ with
	\[
		|S|\leq \exp\pbra{O\pbra{\frac{1}{\eps^2}}}
	\]
	and shifts $\{c_s\}_{s\in S}$ such that
	\[
		\left\|
		\max_{t\in T_m}\bg\cdot t
		-
		\max_{s\in S}\cbra{\bg\cdot s+c_s}
		\right\|_{L^2}
		\leq
		\frac{\eps}{2}
		\cdot
		w(T_m)
		\leq
		\frac{\eps}{2}
		\cdot
		w(T)\,.
	\]
	The shifts are non-negative by the finite case.
	The triangle inequality completes the proof for bounded $T$.
\end{proof}

\subsection{Extensions of \Cref{thm:GP-sparsification}}

We record two extensions of our main result that will be useful in \Cref{sec:apps}. 
The first removes the shifts, at the cost of allowing the approximating centered process to be indexed by a new set rather than by a subset of the original index set. 
A similar centered sparsification result was obtained in \cite[Corollary~23]{de2026sparsifying}, and our proof follows their general strategy of simulating the shifts $\{c_s\}$ from \Cref{thm:GP-sparsification} using auxiliary Gaussian directions. 
In our setting, this reduction requires a few additional checks, including a uniform bound on the magnitude of the shifts and preservation of the $L^2$ error.  
We include the full argument for completeness. 

\begin{corollary}[Centered sparsification]
	\label{cor:centered-sparsification}
	Let $0<\eps<1$ and let $T\sse\R^n$ be bounded.  
	There exists an integer $m\geq n$ and a set $U\sse\R^m$ such that 
	\[
		|U|
		\leq
		\exp\pbra{O\pbra{\frac1{\eps^2}}}
		\qquad\text{and}\qquad 
		\left\|
			\sup_{t\in T}\bg\cdot (t,0_{m-n})
			-
			\max_{u\in U}\bg\cdot u
		\right\|_{L^2(\gamma_m)}
		\leq
		\eps \cdot w(T)\,, 
	\]
	where $(t, 0_{m-n}) \in \R^m$ denotes $t \in \R^n$ viewed as an element of $\R^m$ by concatenating with zeros. 
\end{corollary}  

\begin{proof} 
	As always, we may assume that $w(T) > 0$. 
	We first apply \Cref{thm:GP-sparsification} with error parameter $\eta:=c\eps$ for a sufficiently small universal constant $c>0$. 
	This gives a set $S\sse T$ and non-negative shifts $\{c_s\}_{s\in S}$ such that
	\[
		|S|
		\leq
		\exp\pbra{O\pbra{\frac1{\eps^2}}}
	\]
	and, writing
	\[
		M(g):=\sup_{t\in T}g\cdot t\,,
		\qquad
		F(g):=\max_{s\in S}\cbra{g\cdot s+c_s}\,,
	\]
	we have $\|M-F\|_{L^2(\gamma_n)} \leq \eta\cdot w(T)$. 
	We first obtain an upper bound on the magnitude of each $c_s$. 
	Fix $s\in S$ and set $Z_s(g):=M(g)-g\cdot s$.
	Since $s\in T$, we have $Z_s(g)\geq0$, and
	\[
		\Ex\sbra{Z_s(\bg)}
		=
		w(T)-\Ex[\bg\cdot s]
		=
		w(T)\,.
	\]
	By Markov's inequality, the event $\{Z_s(\bg)\leq2w(T)\}$ has probability at least $1/2$.
	On this event,
	\[
		F(g)-M(g)
		\geq
		g\cdot s+c_s-M(g)
		=
		c_s-Z_s(g)
		\geq
		c_s-2w(T)\,.
	\]
	Hence
	\[
		\eta^2 w(T)^2
		\geq
		\|M-F\|_{L^2}^2
		\geq
		\frac12\pbra{c_s-2w(T)}_+^2\,.
	\]
	Taking $c$ small enough, this implies
	\begin{equation} \label{eq:centered-shift-bound}
		0\leq c_s\leq 3w(T)
		\qquad\text{for every }s\in S\,.
	\end{equation}
	
	We now simulate each shift by a maximum of auxiliary centered Gaussian variables, following \cite[Corollary~23]{de2026sparsifying}.
	Let $A\geq2$ be an integer to be chosen, let $\bz\sim N(0,I_A)$, and define
	\[
		Z_A(\bz):=\max_{1\leq j\leq A}|\bz_j|\,,
		\qquad
		\mu_A:=\Ex\sbra{Z_A}\,,
		\qquad
		\sigma_A^2:=\Var\sbra{Z_A}\,.
	\]
	For each $s\in S$, put
	\[
		\mathrm{Aux}(s)
		:=
		\cbra{
			\pbra{s,\frac{c_s}{\mu_A}e_j},
			\pbra{s,-\frac{c_s}{\mu_A}e_j}
			:
			j\in[A]
		}
		\sse \R^{n+A}\,,
	\]
	and set
	\[
		U:=\bigcup_{s\in S}\mathrm{Aux}(s)\,.
	\]
	Thus $|U|\leq2A|S|$.
	For $(g,z)\in\R^n\times\R^A$, the centered process indexed by $U$ has supremum
	\[
		\widetilde F(g,z)
		:=
		\max_{u\in U}(g,z)\cdot u
		=
		\max_{s\in S}\cbra{
			g\cdot s+\frac{c_s}{\mu_A}Z_A(z)
		}\,.
	\]
	For every fixed $(g,z)$, the elementary inequality
	$\abs{\max_s x_s-\max_s y_s}\leq\max_s\abs{x_s-y_s}$ gives
	\[
		\abs{F(g)-\widetilde F(g,z)}
		\leq
		\max_{s\in S}c_s\cdot
		\abs{1-\frac{Z_A(z)}{\mu_A}}
		\leq
		3w(T)\cdot
		\abs{1-\frac{Z_A(z)}{\mu_A}}\,,
	\]
	where we used \Cref{eq:centered-shift-bound}.
	Consequently,
	\[
		\|F-\widetilde F\|_{L^2(\gamma_{n+A})}
		\leq
		3w(T)\frac{\sigma_A}{\mu_A}\,.
	\]
	A standard estimate for the maximum of $A$ independent standard Gaussian variables gives
	\[
		\frac{\sigma_A}{\mu_A}
		\leq
		\frac{C}{\log A}
	\]
	for a universal constant $C$ (this bound is used in \cite[Corollary~23]{de2026sparsifying}; alternatively, see \cite[Appendix~A.2]{chatterjee2014superconcentration}).
	Choosing $A=\exp\pbra{\frac{C'}{\eps}}$ with $C'$ sufficiently large gives
	\[
		\|F-\widetilde F\|_{L^2(\gamma_{n+A})}
		\leq
		\frac{\eps}{2}w(T)\,.
	\]
	Since $M(g)=\sup_{t\in T}(g,z)\cdot(t,0_A)$ is independent of $z$, the triangle inequality yields
	\[
		\left\|
			\sup_{t\in T}(\bg,\bz)\cdot(t,0_A)
			-
			\max_{u\in U}(\bg,\bz)\cdot u
		\right\|_{L^2(\gamma_{n+A})}
		\leq
		\pbra{\eta+\frac{\eps}{2}}w(T)
		\leq
		\eps w(T)\,,
	\]
	after decreasing $c$ if necessary.
	Finally,
	\[
		|U|
		\leq
		2A|S|
		\leq
		\exp\pbra{O\pbra{\frac1{\eps}}}
		\exp\pbra{O\pbra{\frac1{\eps^2}}}
		=
		\exp\pbra{O\pbra{\frac1{\eps^2}}},
	\]
	which completes the proof with $m=n+A$.
\end{proof}

We also record an extension of \Cref{thm:GP-sparsification} which gives an $L^p$-error variant.  
The proof will require the following moment bound for strongly log-concave distributions. 

\begin{lemma}[Moment bound for strongly log-concave measures] 
\label{lemma:LSI-strongly-log-concave}
	Let $\mu$ be a probability measure on $\R^n$ with density proportional to $e^{-V}$, where $V\in C^2(\R^n)$ satisfies $\nabla^2 V \succeq I$.
	Then there is a universal constant $C>0$ such that for every $u\in\R^n$
	and every $p\geq2$,
	\[
		\left\|
			x\cdot u-\Ex_{\bx\sim\mu}\sbra{\bx\cdot u}
		\right\|_{L^p(\mu)}
		\leq
		C\sqrt p\,\|u\|_2\,.
	\]
\end{lemma}

\begin{proof}
	Since $\nabla^2V\succeq I$, it follows that $\mu$ satisfies a log-Sobolev inequality with a universal constant; see  \cite[Problem~3.19(d)]{rvh-notes}.  
	We thus get subgaussian concentration for Lipschitz functions under $\mu$, as in the proof of \cite[Theorem~3.25]{rvh-notes}.  
	In particular, for every $L$-Lipschitz function
	$F:\R^n\to\R$,
	\[
		\Prx_{\bx\sim\mu}
		\sbra{
			\abs{F(\bx)-\Ex_\mu[F]}\geq t
		}
		\leq
		2\exp\pbra{
			-\frac{c t^2}{L^2}
		}
	\]
	for a universal constant $c>0$.  Applying this to
	$F(x)=x\cdot u$, whose Lipschitz constant is $\|u\|_2$, gives
	\[
		\Pr_{\bx\sim\mu}
		\sbra{
			\abs{\bx\cdot u-\Ex_\mu\sbra{\bx\cdot u}}\geq t
		}
		\leq
		2\exp\pbra{
			-\frac{c t^2}{\|u\|_2^2}
		}\,.
	\]
	Integrating the tail bound,
	\[
		\Ex_\mu
		\sbra{
			\abs{\bx\cdot u-\Ex_\mu\sbra{\bx\cdot u}}^p
		}
		=
		p\int_0^\infty
		t^{p-1}
		\Pr_\mu
		\sbra{
			\abs{\bx\cdot u-\Ex_\mu\sbra{\bx\cdot u}}\geq t
		}
		\,dt
		\leq
		(C\sqrt p\,\|u\|_2)^p\,.
	\]
	Taking $p^{\text{th}}$ roots proves the lemma. 
\end{proof}

We can now prove the $L^p$-error variant of \Cref{thm:GP-sparsification}: 

\begin{corollary}[$L^p$ sparsification]
\label{cor:lp}
	Let $p\geq2$, let $0<\alpha<1$, and let $T\sse\R^n$ be bounded.  
	Then there exists a subset $S\sse T$ and non-negative shifts $\{c_s\}_{s\in S}$ such
	that
	\[
		|S|
		\leq
		\exp\pbra{
			O\pbra{
				\frac{p}{\alpha^2}
			}
		}
		\qquad\text{and}\qquad 
		\left\|
			\sup_{t\in T}\bg\cdot t-\max_{s\in S}\{\bg\cdot s+c_s\}
		\right\|_{L^p(\gamma_n)}
		\leq
		\alpha\cdot w(T)\,.
	\]
	Additionally, there exists an integer $m\geq n$ and a set $U \sse \R^m$ such that  
	\[
		|U|
		\leq
		\exp\pbra{
			O\pbra{
				\frac{p}{\alpha^2}
			}
		}
		\qquad\text{and}\qquad 
		\left\|
			\sup_{t\in T}\bg\cdot(t,0_{m-n})-\sup_{u\in U}\bg\cdot u 
		\right\|_{L^p(\gamma_m)}
		\leq
		\alpha\cdot w(T)\,.
	\]
\end{corollary}

\begin{proof} 
	The shifted statement follows by the same argument establishing
	\Cref{thm:GP-sparsification}, but with the net radius reduced to account for
	$L^p$ error; we explain the only change below. 
	In the notation of the proof of \Cref{lem:soft-active-clustering}, Jensen's inequality gives the $L^p$ analogue of \Cref{eq:jensen-soft-active}:
	\[
		\left\|
			\frac{d}{d\theta}H_\theta
		\right\|_{L^p(\gamma_n)}^p
		\leq
		\sum_{t\in T}
		\rho_t
		\left\|
			\bg\cdot u_t-\Ex_{\mu_t}\sbra{\bg\cdot u_t}
		\right\|_{L^p(\mu_t)}^p\,.
	\]
	As before, each $\mu_t$ has density proportional to $e^{-V_t}$ with
	$\nabla^2V_t\succeq I$.  Therefore
	\Cref{lemma:LSI-strongly-log-concave} gives
	\[
		\left\|
			\bg\cdot u_t-\Ex_{\mu_t}\sbra{\bg\cdot u_t}
		\right\|_{L^p(\mu_t)}
		\leq
		O\pbra{\sqrt{p}\cdot\|u_t\|_2}\,.
	\]
	Taking a net at radius
	\[
		r=\Theta\pbra{\frac{\alpha\cdot w(T)}{\sqrt p}}
	\]
	and applying Sudakov's minoration (\Cref{lem:sudakov}) gives 
	\[
		|S|
		\leq
		\exp\pbra{
			O\pbra{\frac{w(T)^2}{r^2}}
		}
		=
		\exp\pbra{
			O\pbra{\frac{p}{\alpha^2}}
		}\,.
	\]
	The softmax approximation error is handled as in the proof of
	\Cref{thm:GP-sparsification}, by taking $\beta$ sufficiently large.  
	The non-negativity of the shifts also follows from the same argument as before. 

	We now convert this shifted approximator into a centered one, following the argument in \Cref{cor:centered-sparsification}. 
	By decreasing the shifted error parameter by a sufficiently small constant factor, assume that
	\[
		\|M-F\|_{L^p(\gamma_n)}
		\leq
		\eta w(T)\,,
		\qquad
		M(g):=\sup_{t\in T}g\cdot t\,,
		\qquad
		F(g):=\max_{s\in S}\cbra{g\cdot s+c_s}\,,
	\]
	where $\eta>0$ is a sufficiently small constant multiple of $\alpha$.
	As in the proof of \Cref{cor:centered-sparsification}, each shift is bounded by $0\leq c_s\leq 3w(T)$. 
	Let $A\geq2$, let $\bz\sim N(0,I_A)$, and set $Z_A(\bz)$, $\mu_A$, and $\sigma_A$ as before. 
	Define 
	\[
		U
		:=
		\bigcup_{s\in S}
		\cbra{
			\pbra{s,\frac{c_s}{\mu_A}e_j},
			\pbra{s,-\frac{c_s}{\mu_A}e_j}
			:
			j\in[A]
		}
		\sse\R^{n+A}\,.
	\]
	Then $|U|\leq2A|S|$ and
	\[
		\widetilde F(g,z):=\max_{u\in U}(g,z)\cdot u
		=
		\max_{s\in S}\cbra{
			g\cdot s+\frac{c_s}{\mu_A}Z_A(z)
		}\,. 
	\]
	Therefore
	\[
		\abs{F(g)-\widetilde F(g,z)}
		\leq
		3w(T)\abs{1-\frac{Z_A(z)}{\mu_A}}\,.
	\]
	We next estimate the last factor.
	The map $z\mapsto Z_A(z)=\max_{1\leq j\leq A}|z_j|$ is $1$-Lipschitz, since
	\[
		\abs{Z_A(z)-Z_A(z')}
		\leq
		\max_{1\leq j\leq A}\abs{z_j-z'_j}
		\leq
		\|z-z'\|_2\,.
	\]
	Therefore \Cref{prop:lipschitz-concentration} gives, for every $\theta \geq0$,
	\[
		\Prx_{\bz\sim N(0,I_A)}
		\sbra{
			\abs{Z_A(\bz)-\mu_A}\geq \theta 
		}
		\leq
		2\exp\pbra{-\frac{\theta^2}{2}}\,.
	\]
	Integrating this tail bound yields the usual subgaussian moment bound: 
	\begin{align*}
		\Ex\sbra{\abs{Z_A-\mu_A}^p}
		&=
		p\int_0^\infty t^{p-1}
		\Prx\sbra{\abs{Z_A-\mu_A}\geq t}\,dt \\
		&\leq
		2p\int_0^\infty t^{p-1}e^{-t^2/2}\,dt
		\leq
		(C\sqrt p)^p\,.
	\end{align*}
	Hence $\|Z_A-\mu_A\|_{L^p(\gamma_A)} \leq C\sqrt p$. 
	Finally, the standard lower bound on the expected maximum of $A$ independent standard Gaussian variables gives $\mu_A\geq c\sqrt{\log A}$ for $A\geq2$.
	Thus
	\[
		\left\|1-\frac{Z_A}{\mu_A}\right\|_{L^p(\gamma_A)}
		=
		\frac{\|Z_A-\mu_A\|_{L^p(\gamma_A)}}{\mu_A}
		\leq
		C\sqrt{\frac{p}{\log A}}\,.
	\]
	Choosing $A = \exp(C'p/\alpha^2)$ with $C'$ sufficiently large, we get
	\[
		\|F-\widetilde F\|_{L^p(\gamma_{n+A})}
		\leq
		\frac{\alpha}{2}w(T)\,.
	\]
	The triangle inequality then gives the centered $L^p$ approximation, after decreasing the constant in the shifted error parameter.
	The size bound remains
	\[
		|U|
		\leq
		2A|S|
		\leq
		\exp\pbra{
			O\pbra{\frac{p}{\alpha^2}}
		}\,,
	\]
	completing the proof. 
\end{proof}

%% file: sections/applications.tex
\section{Applications of \Cref{thm:GP-sparsification}}
\label{sec:apps}

We now record two applications of \Cref{thm:GP-sparsification}, both following the arguments of \cite{de2026sparsifying} but with our improved sparsifier size yielding exponential improvements in the quantitative bounds. 

\subsection{A Junta Theorem for Norms}
\label{subsec:junta}

We now prove \Cref{thm:norm-junta}, which exponentially improves on \cite[Theorem~2]{de2026sparsifying}. 

\normjuntathm*

The proof follows the same general strategy as in \cite{de2026sparsifying}: first obtain an additive approximation to the norm by a low-dimensional norm, and then use anti-concentration of Gaussian suprema to convert this into a multiplicative approximation.  
A direct implementation of the argument from \cite{de2026sparsifying}, however, leads to a dependence of $\eps^{-5}$ in the exponent.  
The improvement here comes from using the $L^p$-error sparsifier of \Cref{cor:lp} for an appropriate choice of $p$. 

\begin{lemma}[Lemma~26 of~\cite{de2026sparsifying}] 
\label{lem:DNOS-anti}
	Suppose $T \sse \R^n$ is symmetric (that is, $t \in T$ if and only if $-t \in T$) and suppose $w(T) > 0$.  
	Then, for every $\eta>0$,
	\[
		\Prx_{\bg\sim N(0,I_n)}
		\sbra{
			\sup_{t\in T}\bg\cdot t \leq \eta\cdot  w(T)
		}
		\leq
		10\eta \,.
	\]
\end{lemma}

We can now prove \Cref{thm:norm-junta}.

\begin{proof}[Proof of \Cref{thm:norm-junta}]
	Let $B^\circ$ be the dual unit ball of $\psi$, that is 
	\[
		B^\circ
		:=
		\cbra{
			t\in\R^n:
			x\cdot t\leq \psi(x)\text{ for every }x\in\R^n
		}\,.
	\]
	Set $T:=B^\circ$, and so $\psi(x)=\sup_{t\in T} x\cdot t$. 
	Note that $T$ is bounded and symmetric, and that $w(T) = \E_{\bg\sim N(0,I_n)}[\psi(\bg)]$. 
	If $w(T)=0$, then the result is trivial.  
	Otherwise, since the conclusion is invariant under rescaling $\psi$ and $\phi$ by the same positive constant, we may rescale and assume throughout the proof that $w(T)=1$. 

	Set
	\[
		p:=A\log\pbra{\frac{1}{\eps}}
		\qquad\text{and}\qquad
		\alpha:=a\eps^2\,,
	\]
	where $A>0$ is a sufficiently large constant and $a>0$ is a
	sufficiently small constant.  
	Let
	\[
		K_\eps
		:=
		\exp\pbra{
			C_0\frac{\log(1/\eps)}{\eps^4}
		}
	\]
	for a sufficiently large universal constant $C_0$.
	If $n\leq K_\eps$, then the theorem is trivial by taking $\phi=\psi$.
	We therefore assume $n>K_\eps$.
	
	Apply the centered part of \Cref{cor:lp} with error parameter $\alpha$.
	Since $T$ is symmetric, the resulting centered index set may be taken symmetric: in the shifted step, one first replaces the shifted sparsifier by its symmetrization, using the same shift for $s$ and $-s$, and then the auxiliary-coordinate construction preserves symmetry.
	Moreover, the auxiliary directions used in that construction span only
	$\exp(O(p/\alpha^2))$ dimensions, so after increasing $C_0$ if necessary the assumption $n>K_\eps$ lets us realize those directions inside the orthogonal complement of the shifted sparsifier in the original ambient space $\R^n$.
	Thus there is a symmetric set $S\sse \R^n$ such that
	\begin{equation} \label{eq:nuonuo}
		\left\|
			\sup_{t\in T} \bg\cdot t
			-
			\sup_{s\in S} \bg\cdot s
		\right\|_{L^p(\gamma_n)}
		\leq
		a\eps^2
	\end{equation}
	Moreover,
	\[
		|S|
		\leq
		\exp\pbra{
			O\pbra{
				\frac{\log(1/\eps)}{\eps^4}
			}
		}\,.
	\]
	Let $E':=\operatorname{span}(S)$.
	Then $\dim(E')\leq |S|\leq K_\eps$, after increasing $C_0$ if necessary.
	Define $\phi(x):=\sup_{s\in S}x\cdot s$.
	Since $S$ is symmetric, $\phi$ is a $\dim(E')$-junta norm.

	We now convert the additive $L^p$-guarantee from \Cref{eq:nuonuo} into a multiplicative one. 
	Let $b>0$ and $C>0$ be constants to be fixed later.
	Note that 
	\begin{align*}
		\Prx_{\bg\sim N(0,I_n)}
		\sbra{
			\abs{
				\sup_{t\in T}\bg\cdot t
				-
				\sup_{s\in S}\bg\cdot s
			}
			>
			b\eps^2
		} 
		&=
		\Prx_{\bg\sim N(0,I_n)}
		\sbra{
			\abs{
				\sup_{t\in T}\bg\cdot t
				-
				\sup_{s\in S}\bg\cdot s
			}^p
			>
			(b\eps^2)^p
		} \\
		&\leq \pbra{\frac{a}{b}}^p \tag{Markov's inequality and \Cref{eq:nuonuo}}  \\
		&\leq \frac{\eps}{2} 
	\end{align*}
	by first fixing $C\leq 1/20$ and $b\leq C/2$, and then taking $a>0$ sufficiently small and $A$ sufficiently large.
	On the other hand, by \Cref{lem:DNOS-anti}, 
	\[
		\Prx_{\bg\sim N(0,I_n)}
		\sbra{
			\sup_{t\in T}\bg\cdot t
			\leq
			C\eps
		}
		\leq
		10C\eps\,.
	\]
	With the above choice $C\leq 1/20$, this probability is at most $\eps/2$.  
	Thus the two exceptional events have total probability at most $\eps$.  
	Outside them, we have
	\[
		\sup_{t\in T}g\cdot t
		\geq
		C\eps 
		\qquad\text{and}\qquad 
		\abs{
			\sup_{t\in T}g\cdot t
			-
			\sup_{s\in S}g\cdot s
		}
		\leq
		b\eps^2\,.
	\]
	It follows that 
	\[
		\abs{
			\sup_{t\in T}g\cdot t
			-
			\sup_{s\in S}g\cdot s
		}
		\leq
		\frac{\eps}{2}\cdot 
		\sup_{t\in T}g\cdot t\,.
	\]
	Since $\sup_{t\in T}g\cdot t=\psi(g)$ and $\sup_{s\in S}g\cdot s=\phi(g)$, we obtain
	\[
		\pbra{1-\frac{\eps}{2}}\,\psi(g)
		\leq
		\phi(g)
		\leq
		\pbra{1+\frac{\eps}{2}}\,\psi(g)\,.
	\]	
	Since $0<\eps<1/2$ and both quantities are positive outside the exceptional events, this implies 
	\[
		\pbra{1 + \frac{\eps}{2}}^{-1}
		\leq
		\frac{\psi(g)}{\phi(g)}
		\leq
		\pbra{1 - \frac{\eps}{2}}^{-1}\,.
	\]
	The stated bound now follows from the easy estimates $(1+\eps/2)^{-1}\geq1-\eps$ and $(1-\eps/2)^{-1}\leq1+\eps$. 
\end{proof} 

\subsection{Learning, Testing, and Approximating Convex Sets} 
\label{subsec:apps-polytope}

We next use \Cref{thm:GP-sparsification} to sparsify intersections of narrow halfspaces.  
As in \cite{de2026sparsifying}, the resulting geometric approximation theorem yields learning and testing guarantees when combined with known algorithms for intersections of halfspaces. 

Recall that for measurable sets $K,L\sse\R^n$, we write
\[
	\dG(K,L)
	:=
	\Prx_{\bg\sim N(0,I_n)}
	\sbra{\bg\in K\,\triangle\,L}
\]
for their Gaussian distance.
Additionally, recall that a convex set $K\sse\R^n$ has geometric width at most $r$ if it can be written as
\[
	K
	=
	\bigcap_{t\in T}
	\cbra{x\in\R^n:t\cdot x\leq r_t}\,,
	\qquad
	T\sse\S^{n-1}
	\qquad\text{and}\qquad
	\abs{r_t}\leq r
	\quad\text{for all }t\in T\,.
\]

\subsubsection{Approximating Intersections of Narrow Halfspaces}

We first prove the approximation theorem in the exact-width case, where every defining halfspace has the same distance from the origin; this is the setting of \cite[Lemma~28]{de2026sparsifying}. 
The proof follows the same threshold-conversion argument as in \cite{de2026sparsifying}, but uses the $L^p$ sparsifier from \Cref{cor:lp}. 
Using \Cref{thm:GP-sparsification} would give an extra factor of $\eps^{-1}$ in the exponent; taking $p\asymp\log(1/\eps)$ removes this loss. 

\begin{lemma}[Exact-width halfspace sparsification]
	\label{lem:exact-width-polytope}
	Let $r\geq1$ and $0<\eps<1/2$.
	Suppose
	\[
		T\sse \frac1r \S^{n-1}
		\qquad\text{and}\qquad
		K=
		\bigcap_{t\in T}
		\cbra{x\in\R^n:t\cdot x\leq 1}\,.
	\]
	Then there is a set $L\sse\R^n$, which is an intersection of at most
	\[
		\exp\pbra{
			O\pbra{
				\frac{\log(1/\eps)(r+\sqrt{\log(1/\eps)})^4}{\eps^2}
			}
		}
	\]
	halfspaces, such that $\dG(K,L)\leq\eps$.
\end{lemma}

\begin{proof}
	We may assume that $\Prx_{\bg\sim N(0,I_n)}[\bg\in K]\in(\eps,1-\eps)$.  
	Indeed, if this measure is at most $\eps$, the empty set is an $\eps$-approximator; if it is at least $1-\eps$, then all of $\R^n$ is an $\eps$-approximator. 
	Since
	$K=
		\cbra{
			x\in\R^n:
			\sup_{t\in T}x\cdot t\leq 1
		},
	$ 
	it follows that 
	\[
		\Prx_{\bg\sim N(0,I_n)}
		\sbra{
			\sup_{t\in T}\bg\cdot t\leq 1
		}
		\geq
		\eps\,.
	\]
	The function $x\mapsto \sup_{t\in T}x\cdot t$ is $(1/r)$-Lipschitz. 
	Thus, if $w(T)\geq1$, concentration of Lipschitz functions of Gaussians (\Cref{prop:lipschitz-concentration}) gives 
	\[
		\eps
		\leq
		\Prx_{\bg\sim N(0,I_n)}
		\sbra{
			\sup_{t\in T}\bg\cdot t\leq 1
		}
		\leq
		2\exp\pbra{
			\frac{-r^2(w(T)-1)^2}{2}
		}\,.
	\]
	Consequently, we have the following bound for all $T$ as in the statement of \Cref{lem:exact-width-polytope}: 
	\begin{equation}
		\label{eq:polytope-width-bound}
		w(T)
		\leq
		1+\frac{1}{r}\sqrt{2\log\pbra{\frac{2}{\eps}}}\,.
	\end{equation}

	We next use an anti-concentration inequality due to Chernozhukov, Chetverikov, and Kato~\cite{CCK-2} which says that if $U\sse\S^{n-1}$, then for all $\theta\in\R$ and $\delta>0$ we have 
	\[
		\Prx_{\bg\sim N(0,I_n)}
		\sbra{
			\abs{
				\sup_{u\in U}\bg\cdot u-\theta
			}
			\leq
			\delta
		}
		\leq
		4\delta\pbra{1+w(U)}\,.
	\]
	Applying this to $rT\sse\S^{n-1}$ gives, for every $\tau>0$,
	\begin{align}
		\Prx_{\bg\sim N(0,I_n)}
		\sbra{
			\abs{
				\sup_{t\in T}\bg\cdot t-1
			}
			\leq
			\tau
		}
		&=
		\Prx_{\bg\sim N(0,I_n)}
		\sbra{
			\abs{
				\sup_{u\in rT}\bg\cdot u-r
			}
			\leq
			r\tau
		} \nonumber \\
		&\leq
		4r\tau\pbra{1+w(rT)} \nonumber \\
		&\leq
		4r\tau\pbra{2r+\sqrt{2\log\pbra{\frac{2}{\eps}}}}\,,
		\label{eq:polytope-anti-conc}
	\end{align}
	where the last step uses \Cref{eq:polytope-width-bound} and $r\geq1$.

	Put
	\[
		A:=4r\pbra{2r+\sqrt{2\log\pbra{\frac{2}{\eps}}}}
		\qquad\text{and}\qquad
		\tau:=\frac{\eps}{2A}\,.
	\]
	Let $p := C\log(1/\eps)$ for a sufficiently large universal constant $C > 0$, and let $\eta := c\tau$ for a sufficiently small universal constant $c > 0$. 
	We now apply \Cref{cor:lp} with absolute error $\eta$. 
	More precisely, if $\eta < w(T)$, we apply it with accuracy parameter $\eta/w(T)$; otherwise we apply it with any fixed constant accuracy parameter. 
	In either case, we obtain $S\sse T$ and non-negative shifts $\{c_s\}_{s\in S}$ such that
	\[
		|S|
		\leq
		\exp\pbra{
			O\pbra{
				p+\frac{p\cdot w(T)^2}{\eta^2}
			}
		}
	\]
	and
	\[
		\left\|
			\sup_{t\in T}\bg\cdot t
			-
			\sup_{s\in S}\{\bg\cdot s+c_s\}
		\right\|_{L^p(\gamma_n)}
		\leq
		\eta\,.
	\]
	Define the polytope
	\[
		L
		:=
		\cbra{x\in\R^n: \sup_{s\in S} \cbra{x\cdot s + c_s} \leq 1}
		=
		\bigcap_{s\in S}
		\cbra{x\in\R^n:s\cdot x\leq1-c_s}\,.
	\]
	By Markov's inequality,
	\[
		\Prx_{\bg\sim N(0,I_n)}
		\sbra{
			\abs{
				\sup_{t\in T}\bg\cdot t
				-
				\sup_{s\in S} \cbra{\bg\cdot s + c_s}
			}
			>
			\tau
		}
		\leq
		\pbra{\frac{\eta}{\tau}}^p
		\leq
		\frac{\eps}{2}\,,
	\]
	thanks to appropriate choices of $c$ and $C$. 
	On the other hand, \Cref{eq:polytope-anti-conc} and the choice of $\tau$ give
	\[
		\Prx_{\bg\sim N(0,I_n)}
		\sbra{
			\abs{
				\sup_{t\in T}\bg\cdot t-1
			}
			\leq
			\tau
		}
		\leq
		\frac{\eps}{2}\,.
	\]
	Outside these two exceptional events, which happen with probability $\eps$, the two thresholded decisions agree:
	\[
		\sup_{t\in T}g\cdot t\leq 1
		\qquad\text{if and only if}\qquad
		\sup_{s\in S} \cbra{g\cdot s + c_s} \leq 1\,.
	\]
	It follows that $\dG(K,L)\leq\eps$.

	It remains to bound the number of halfspaces. 
	For notational convenience we set $a := \sqrt{2\log(2/\eps)}$. 
	By \Cref{eq:polytope-width-bound}, we have 
	\[
		w(T)
		\leq
		1+\frac{a}{r}
		=
		\frac{r+a}{r}\,.
	\]
	Also,
	\[
		\frac1{\tau^2}
		=
		O\pbra{
			\frac{r^2(2r+a)^2}{\eps^2}
		}\,.
	\]
	Since $\eta=c\tau$, we get
	\[
		\frac{p\cdot w(T)^2}{\eta^2}
		=
		O\pbra{
			\frac{\log(1/\eps)(r+a)^2(2r+a)^2}{\eps^2}
		}
		\leq
		O\pbra{
			\frac{\log(1/\eps)(r+a)^4}{\eps^2}
		}\,.
	\]
	The additional $O(p)$ term is absorbed into this bound for $r\geq1$ and $0<\eps<1/2$.
	Recalling that $a=\sqrt{2\log(2/\eps)}$ gives
	\[
		|S|
		\leq
		\exp\pbra{
			O\pbra{
				\frac{\log(1/\eps)(r+\sqrt{\log(1/\eps)})^4}{\eps^2}
			}
		}
	\]
	as desired. 
\end{proof}

The following is implicit in the proof of Theorem~3 of \cite{de2026sparsifying}: If every exact-width body of width $R$ admits $\delta$-approximation by $M(R,\delta)$ halfspaces, then every convex set of geometric width at most $r$ admits $O(\delta)$-approximation by $M(\Theta(r),\Theta(\delta))$ halfspaces. 
Applying this reduction to \Cref{lem:exact-width-polytope} yields \Cref{thm:polytope-approximation}. 

\polytopeapproxthm*


\subsubsection{Algorithmic Applications}

We now record algorithmic consequences of \Cref{thm:polytope-approximation}, following
Section~5 of~\cite{de2026sparsifying} with the improved value of $\newalpha$. 
For $r\geq 1$, let $\Conv(r)$ denote the class of indicators of convex sets in $\R^n$ of geometric width at most $r$.
For Boolean functions $f,h:\R^n\to\zo$, write
\[
	\dG(f,h)
	:=
	\Prx_{\bg\sim N(0,I_n)}
	\sbra{
		f(\bg)\neq h(\bg)
	}
	\qquad\text{and}\qquad
	\dG(f,\calC)
	:=
	\inf_{h\in\calC}\dG(f,h)\,.
\]
Let $\calH_k$ denote the class of intersections of at most $k$ halfspaces in $\R^n$. 
Thus, by \Cref{thm:polytope-approximation}, every function in $\Conv(r)$ is $\eps$-close in Gaussian distance to a function in $\calH_{\newalpha}$.
Equivalently, writing 
\[
	\OPT_r(f)
	:=
	\dG(f,\Conv(r))
	\qquad\text{and}\qquad 
	\OPT'_k(f) 
	:= 
	\dG(f, \calH_k)\,,
\]
we have, for every Boolean function $f:\R^n\to\zo$, 
\begin{equation} \label{eq:stitch}
	\OPT'_{\newalpha}(f)
	\leq
	\OPT_r(f)+\eps
\end{equation}
thanks to \Cref{thm:polytope-approximation}. 

\paragraph{Agnostic Learning.}

Klivans, O'Donnell, and Servedio~\cite{KOS:08} gave an algorithm running in time $n^{O(\log(k)/\eps^4)}$ that agnostically learns intersections of $k$ halfspaces over Gaussian space from independent labeled examples drawn from $N(0,I_n)$.  
More precisely, given examples $\bx\sim N(0,I_n)$ labeled by an arbitrary target function $f:\R^n\to\zo$, their algorithm outputs, with high probability, a hypothesis $h:\R^n\to\zo$ satisfying
\[
    \dG(f,h)
    \leq
    \OPT'_k(f)+\eps\,.
\] 
Combining this with \Cref{eq:stitch} immediately gives the following. 

\begin{corollary}[Agnostic learning from examples]
	\label{cor:polytope-learning-examples}
	Let $r\geq1$, let $0<\eps<1/2$, and let $f:\R^n\to\zo$ be an unknown Boolean function. 
	There is an algorithm which, given access to independent labeled examples
	$(\bx,f(\bx))$ with $\bx\sim N(0,I_n)$, runs in time
	\[
		n^{
			O\pbra{
				\log(\alpha(r,\eps/2))/\eps^4
			}
		}
		=
		n^{\wt{O}(r^4/\eps^6)}
	\]
	and, with probability $9/10$, outputs a hypothesis $h:\R^n\to\zo$ satisfying 
	\[
		\dG(f,h)
		\leq
		\OPT_r(f)+\eps\,.
	\]
\end{corollary}

Sometimes, one also has the ability to query the unknown target function $f$ on inputs of their choice, rather than merely receiving passively labeled examples. 
Diakonikolas, Kane, Kontonis, Tzamos, and Zarifis~\cite{DKKTZ23} showed that this additional power can be exploited to give a substantially more efficient agnostic learning algorithm for intersections of $k$ halfspaces: their algorithm runs in time
\[
	\poly(n)\cdot 2^{\poly(\log(k)/\eps)}\,,
\]
assuming black-box query access to the target function. Combining their result with \Cref{eq:stitch} gives:

\begin{corollary}[Agnostic learning with queries]
	\label{cor:polytope-learning-queries}
	Let $r\geq1$, let $0<\eps<1/2$, and let $f:\R^n\to\zo$ be an unknown
	Boolean function. 
	There is an algorithm which, given black-box query access to $f$, runs in
	time
	\[
		\poly(n)\cdot
		2^{
			\poly\pbra{
				\log(\alpha(r,\eps/2))/\eps
			}
		}
		= \poly(n)\cdot 2^{\poly(r, 1/\eps)}
	\]
	and, with probability $9/10$, outputs a hypothesis $h:\R^n\to\zo$ satisfying
	\[
		\dG(f,h)
		\leq
		\OPT_r(f)+\eps\,.
	\]
\end{corollary}

\paragraph{Property Testing.}

We also obtain a tolerant testing algorithm (equivalently, a distance-estimation algorithm) for $\Conv(r)$ with dimension-free query complexity.  
In particular, we will prove the following: 

\begin{corollary}[Tolerant testing]
	\label{cor:polytope-tolerant-testing}
	Let $r\geq1$ and let $0\leq\eps_1<\eps_2<1$.
	Put
	\[
		\Delta := \eps_2-\eps_1
		\qquad\text{and}\qquad
		k := \alpha(r,\Delta/4)\,.
	\]
	There is a black-box query algorithm which makes $k^{\poly(\log k,1/\Delta)} = 2^{\poly(r,1/\Delta)}$ queries to an unknown target function $f:\R^n\to\zo$, and with probability at least $9/10$ has the following guarantee:
	\begin{itemize}
		\item If $\OPT_r(f)\leq \eps_1$, then it outputs ``accept.''
		\item If $\OPT_r(f)\geq \eps_2$, then it outputs ``reject.''
	\end{itemize}
\end{corollary}

Before giving the proof, we recall a tolerant testing theorem of De, Mossel, and Neeman~\cite{DMN21} for \emph{linear $k$-juntas}. 
Unlike in the agnostic learning corollaries, it is not enough to combine \Cref{eq:stitch} with a tolerant tester for the full class $\calH_k$ of intersections of $k$ halfspaces.
Indeed, a function may be close to an intersection of few halfspaces while still being far from $\Conv(r)$.
The key point is that the tolerant tester of~\cite{DMN21} applies not only to the full class of intersections of $k$ halfspaces, but to arbitrary subclasses of \emph{linear $k$-juntas} with bounded Gaussian surface area. 

Recall that a Boolean function $h:\R^n\to\zo$ is a \emph{linear $k$-junta} if there are vectors $u_1,\ldots,u_k\in\R^n$ and a Boolean function $g:\R^k\to\zo$ such that
\[
	h(x)=g(u_1\cdot x,\ldots,u_k\cdot x)\,.
\]
In particular, every intersection of at most $k$ halfspaces is a linear $k$-junta.
We use the following result of De, Mossel, and Neeman~\cite[Theorem~1.2]{DMN21}: if $\calC$ is any class of linear $k$-juntas, each of Gaussian surface area at most $s$, then for every $0\leq\rho_1<\rho_2<1$, there is a black-box query algorithm distinguishing $\dG(f,\calC)\leq \rho_1$ from $\dG(f,\calC)\geq \rho_2$ with query complexity $k^{\poly(s,1/(\rho_2-\rho_1))}$. 
We will apply this theorem to the subclass of $\calH_k$ consisting of those intersections of $k$ halfspaces that are themselves close to $\Conv(r)$.

\begin{proof}[Proof of~\Cref{cor:polytope-tolerant-testing}]
	Set parameters 
	\[
		\eta := \Delta/4
		\qquad\text{and}\qquad
		k := \alpha(r,\eta)\,.
	\]
	Let $\calH_k$ denote the class of intersections of at most $k$ halfspaces, and define the subclass
	\[
		\calQ_{r,\eta,k}
		:=
		\cbra{
			h\in\calH_k:
			\dG(h,\Conv(r))\leq \eta
		}\,.
	\]
	Note that this is a subclass of the class of linear $k$-juntas.
	Moreover, every $h\in\calQ_{r,\eta,k}$ is an intersection of at most $k$ halfspaces, and hence has Gaussian surface area at most $O(\sqrt{\log k})$ thanks to Nazarov's bound on the Gaussian surface area of intersections of $k$ halfspaces~\cite[Theorem~20]{KOS:08}. 
	Therefore we can apply \cite[Theorem~1.2]{DMN21} to $\calQ_{r,\eta,k}$ with surface-area parameter $s=O(\sqrt{\log k})$. 
	In particular, since $O(\sqrt{\log k})$ may be absorbed into $\poly(\log k)$, it gives a tolerant tester for $\calQ_{r,\eta,k}$ with query complexity
	\[
		k^{\poly(\log k,1/\Delta)}\,.
	\]
	We apply this tester with thresholds $\rho_1 := \eps_1+\eta$ and $\rho_2 := \eps_2-\eta$. 
	These thresholds are separated, since
	\[
		\rho_1
		=
		\eps_1+\frac{\Delta}{4}
		<
		\eps_2-\frac{\Delta}{4}
		=
		\rho_2\,.
	\]
	
	We now verify that this tester has the desired guarantee for $\Conv(r)$.
	First suppose that $\OPT_r(f)=\dG(f,\Conv(r))\leq \eps_1$. 
	Then there is some $g\in\Conv(r)$ such that $\dG(f,g)\leq \eps_1$.
	By \Cref{thm:polytope-approximation}, there is some $h\in\calH_k$ such that $\dG(g,h)\leq \eta$. 
	In particular, $h\in\calQ_{r,\eta,k}$.
	Therefore
	\[
		\dG(f,\calQ_{r,\eta,k})
		\leq
		\dG(f,h)
		\leq
		\dG(f,g)+\dG(g,h)
		\leq
		\eps_1+\eta
		=
		\rho_1\,.
	\]
	Thus the tester accepts in the completeness case. 
	Conversely, suppose $\OPT_r(f)=\dG(f,\Conv(r))\geq \eps_2$. 
	Let $h\in\calQ_{r,\eta,k}$.
	By definition of $\calQ_{r,\eta,k}$, we have $\dG(h,\Conv(r))\leq \eta$. 
	Hence, by the triangle inequality,
	\[
		\dG(f,h)
		\geq
		\dG(f,\Conv(r))-\dG(h,\Conv(r))
		\geq
		\eps_2-\eta
		=
		\rho_2\,.
	\]
	Taking the infimum over $h\in\calQ_{r,\eta,k}$, we get $\dG(f,\calQ_{r,\eta,k})
		\geq
		\rho_2$.
	Thus the tester rejects in the soundness case.
\end{proof}

%% file: sections/lower-bounds.tex
\section{Lower Bounds}
\label{sec:lbs}

In this section, we establish lower bounds complementing
\Cref{thm:GP-sparsification,thm:polytope-approximation}. 

\subsection{Gaussian Process Sparsification}
\label{subsec:lb-gp}

We first prove a lower bound for Gaussian process sparsification, showing that the sparsifier size in \Cref{thm:GP-sparsification} is optimal up to constant factors in the exponent. 
In fact, this lower bound holds even for sparsifiers with arbitrary real shifts, rather than only non-negative shifts. 

\widthonegplbthm*

The key step towards \Cref{thm:width-one-gp-lb} is the following lemma, which itself is a consequence of standard bounds on the measure of spherical caps. 

\begin{lemma}
\label{lem:continuous-sphere-lb}
	There are universal constants $c_0,\kappa>0$ such that the following holds
	for all sufficiently large $n$.  
	Let $u_1,\ldots,u_m\in \S^{n-1}$ and $a_1,\ldots,a_m\in\R$, and define
	\[
		h(g)
		:=
		\max_{1\leq j\leq m}
		\cbra{a_j+\frac{1}{\sqrt n} g\cdot u_j}\,.
	\]
	If
	\[
		\left\|
			\frac{\|\bg\|_2}{\sqrt n}-h(\bg)
		\right\|_{L^2}
		\leq 
		\frac{\kappa}{\sqrt n}\,,
	\]
	then $m\geq \exp(c_0 n)$.
\end{lemma}

\begin{proof}
	Fix $\theta\in \S^{n-1}$ and restrict $h$ to the rescaled ray
	$\{\sqrt n r\theta:r\geq 0\}$.  Thus
	\[
		h_\theta(r)
		:=
		h(\sqrt n r\theta)
		=
		\max_{1\leq j\leq m}\cbra{a_j+r\,\theta\cdot u_j}\,.
	\]
	Let
	\[
		\rho(\theta):=\max_{1\leq j\leq m}\theta\cdot u_j
		\qquad\text{and}\qquad
		A:=\cbra{\theta\in \S^{n-1}:\rho(\theta)\leq \frac{9}{10}}\,.
	\]
	Note that for every $\theta\in A$, every affine function appearing in $h_\theta$ has slope at most $9/10$. 
	Hence, for all $r_1<r_2$,
	\[
		h_\theta(r_2)-h_\theta(r_1)
		\leq
		\frac{9}{10}(r_2-r_1)\,.
	\]
	Equivalently, if $e_\theta(r):=h_\theta(r)-r$, then
	\begin{equation}
	\label{eq:abs-e-theta-lb}
		e_\theta(r_2)-e_\theta(r_1)
		\leq
		-\frac{1}{10}(r_2-r_1)\,,
		\qquad\text{and therefore}\qquad 
		\abs{e_\theta(r_2)-e_\theta(r_1)}
		\geq
		\frac{1}{10}\abs{r_2-r_1}\,.
	\end{equation}

	Write $\bg\sim N(0, I_n)$ as $\bg=\bR\,\btheta$, where $\bR \sim \chi(n)$ and $\btheta\sim\S^{n-1}$ according to the Haar measure.  
	Let $\br:=\bR/\sqrt n$.  For each $\theta\in A$, using two
	independent copies $\br_1,\br_2$ of $\br$, we get
	\begin{align*}
		\Ex_{\br}\sbra{e_\theta(\br)^2}
		&\geq
		\Varx_{\br}\sbra{e_\theta(\br)} \\
		&=
		\frac{1}{2}
		\Ex_{\br_1,\br_2}
		\sbra{
			\abs{e_\theta(\br_1)-e_\theta(\br_2)}^2
		} \\
		&\geq
		\frac{1}{100}
		\Var\sbra{\frac{\bR}{\sqrt n}}\,,
	\end{align*}
	where the last step uses \Cref{eq:abs-e-theta-lb}.  
	Since $\bR\sim \chi(n)$, we have $\Var[\bR]\to 1/2$ as $n\to\infty$.
	Thus, for all sufficiently large $n$,
	\[
		\Var\sbra{\frac{\bR}{\sqrt n}}
		\geq
		\frac{1}{4n}\,,
		\qquad\text{and so for every}~\theta\in A,\qquad 
		\Ex_{\br}\sbra{e_\theta(\br)^2}
		\geq
		\frac{1}{400n}\,.
	\]
	
	Integrating over $\theta$ gives
	\begin{align*}
		\left\|
			\frac{\|\bg\|_2}{\sqrt n}-h(\bg)
		\right\|_{L^2}^2
		&=
		\Ex_{\btheta\sim\S^{n-1}}
		\sbra{
			\Ex_{\br}
			\sbra{
				\pbra{\br-h_{\btheta}(\br)}^2
			}
		} \\
		&=
		\Ex_{\btheta\sim\S^{n-1}}
		\sbra{
			\Ex_{\br}
			\sbra{e_{\btheta}(\br)^2}
		} \\
		&\geq
		\Prx_{\btheta\sim\S^{n-1}}\sbra{\btheta\in A}\cdot \frac{1}{400n}\,.
	\end{align*}
	If the $L^2$ error is at most $\kappa/\sqrt n$, then $\Pr_{\btheta}\sbra{\btheta\in A} \leq 400\kappa^2$. 
	Choosing $\kappa>0$ sufficiently small, we may assume that this probability is at most $1/2$. Hence
	\[
		\Prx_{\btheta\sim\S^{n-1}}
		\sbra{
			\rho(\btheta)>\frac{9}{10}
		}
		\geq
		\frac{1}{2}\,.
	\]

	On the other hand,
	\[
		\cbra{\theta:\rho(\theta)>\frac{9}{10}}
		\subseteq
		\bigcup_{j=1}^m
		\cbra{\theta:\theta\cdot u_j>\frac{9}{10}}\,.
	\]
	Each set on the right is a spherical cap cut off by the hyperplane $\theta\cdot u_j=9/10$. 
	By a standard spherical cap estimate \cite[Lemma~2.2]{ball1997elementary}, for every fixed $u\in\S^{n-1}$ and every fixed $a\in(0,1)$, the Haar measure of $\{\theta:\theta\cdot u>a\}$ is at most $\exp(-c(a)n)$; in particular, its Haar measure is at most $\exp(-c'n)$ for a universal constant $c'>0$.
	Therefore,
	\[
		\frac{1}{2}
		\leq
		m\exp(-c'n),
	\]
	which implies $m\geq \exp(c_0 n)$ for a universal constant $c_0>0$. 
\end{proof}

We now return to the proof of the main lower bound.

\begin{proof}[Proof of \Cref{thm:width-one-gp-lb}]
	The preceding lemma gives a lower bound in dimension $n$ at error scale
	$n^{-1/2}$.  
	We take $n \asymp \eps^{-2}$ and then rescale the corresponding sphere process to have Gaussian width exactly one. 

	Fix $0<\eps<\eps_0$ and set
	\[
		n=\floor{\pbra{\frac{\kappa}{\eps}}^2}\,,
	\]
	where $\eps_0>0$ is small enough that $n$ is above the threshold in \Cref{lem:continuous-sphere-lb}.  
	In particular, we have $n \geq \kappa^2/(2\eps^2)$. 
	Define 
	\[
		T_n:=\frac{1}{\sqrt n}\S^{n-1}\,,
		\qquad\text{and}\qquad 
		F_n(g):=\sup_{t\in T_n}g\cdot t
		=
		\frac{\|g\|_2}{\sqrt n}\,.
	\]
	Let $w_n := w(T_n)$, and it is readily checked (using Jensen's inequality) that $w_n \leq 1$. 
	Finally, define 
	\[
		T_\eps := \frac{1}{w_n} T_n\,,
		\qquad\text{and note that}\qquad 
		w(T_\eps) = 1\,,
		\quad 
		\sup_{t\in T_\eps} g\cdot t = \frac{1}{w_n} F_n(g)\,.
	\]

	Suppose $S \sse T_\eps$ and $h(g) := \max_{s\in S} {g\cdot s + b_s}$ is an $\eps$-approximator in $L^2(\gamma_n)$ to $\sup_{t\in T_\eps} g\cdot t$. 
	Multiplying by $w_n$ gives $h'(g) := w_n\,h(g)$. 
	Thus $h'$ has $|S|$ affine pieces with directions in $w_nS\subseteq T_n$ and additionally $\|h' - F_n\|_{L^2} \leq w_n\eps \leq \eps$. 
	Since $\eps \leq \kappa/\sqrt{n}$, it follows from \Cref{lem:continuous-sphere-lb} that 
	\[
		|S| \geq \exp(c_0n)\,.
	\]
Since $n \geq \kappa^2/(2\eps^2)$, this implies that $|S| \geq \exp\pbra{\Omega(1/\eps^2)}$ as desired. 
\end{proof}

\begin{remark}
\label{rem:l1-sphere-lb}
The preceding lower bound also holds for $L^1$-error approximators, after replacing the one-dimensional radial estimate in the proof of \Cref{lem:continuous-sphere-lb}.  
Indeed, fix $\theta\in A$ and write $\br=\bR/\sqrt n$.  
By \Cref{eq:abs-e-theta-lb}, the function $e_\theta(r)=h_\theta(r)-r$ is decreasing at rate at least $1/10$: for $r_1<r_2$,
\[
    e_\theta(r_1)-e_\theta(r_2)
    \geq
    \frac{1}{10}(r_2-r_1).
\]
Let $m_{\br}$ be a median of $\br$. 
Since $e_\theta$ is decreasing, $e_\theta(m_{\br})$ is a median of $e_\theta(\br)$. 
Using the fact that medians minimize $L^1$ distance, for every constant $q$ we have 
\[
    \Ex\sbra{\abs{e_\theta(\br)-q}}
    \geq
    \Ex\sbra{\abs{e_\theta(\br)-e_\theta(m_{\br})}}
    \geq
    \frac{1}{10}\Ex\sbra{\abs{\br-m_{\br}}}
    \gtrsim 
    \frac{1}{\sqrt n}\,,
\]
where the final inequality uses standard asymptotics for the $\chi(n)$ distribution.  
In particular, taking $q=0$ yields, for every $\theta\in A$,
\[
    \Ex_{\br}\sbra{\abs{\br-h_\theta(\br)}}
    =
    \Ex_{\br}\sbra{\abs{e_\theta(\br)}}
    \gtrsim
    \frac{1}{\sqrt n}.
\]
Integrating over $\theta$ gives
\[
    \left\|
        \frac{\|\bg\|_2}{\sqrt n}-h(\bg)
    \right\|_{L^1}
    \gtrsim
    \Prx_{\btheta\sim\S^{n-1}}[\btheta\in A]\cdot \frac{1}{\sqrt n}\,.
\]
Thus an $L^1$ error bound of order $\kappa/\sqrt n$, with $\kappa>0$ sufficiently small, forces $\Pr[\btheta\in A]\leq 1/2$.
Equivalently, $\Pr[\rho(\btheta) > 9/10]  \geq 1/2$. 
The same spherical-cap covering argument used above then gives $m\geq \exp(c_0 n)$. 
After the same rescaling with $n\asymp \eps^{-2}$, every $L^1$-approximator of the width-one sphere process with error at most $\eps$ requires $\exp(\Omega(1/\eps^2))$ pieces, even with arbitrary real shifts. 
Combined with the $L^2$ upper bound of \Cref{thm:GP-sparsification}, which is stronger than an $L^1$ bound, this gives the optimal $L^1$ sparsifier size $\exp(\Theta(1/\eps^2))$.
\end{remark}

\subsection{Polyhedral Approximation}
\label{subsec:polytope-lb}

We next show that the $n$-dimensional cube of Gaussian measure $1/2$ cannot be approximated, in Gaussian distance, by an intersection of $o(n)$ halfspaces. 

\cubepolytopelb*

As mentioned in \Cref{subsec:intro-apps-polytope}, this implies that the dependence on $r$ in \Cref{thm:polytope-approximation} must be at least $\exp(\Omega(r^2))$.  
Indeed, $C_n$ has geometric width $r_n$. Moreover, since $\gamma_n(C_n)=1/2$, we have $\gamma_1([-r_n,r_n])=2^{-1/n},$
and hence $Pr[|\bg_1|>r_n]=1-2^{-1/n}=\Theta(1/n)$.
By a standard univariate Gaussian tail bound~\cite[Proposition~2.1.2]{vershynin2018high}, this implies $r_n^2=\Theta(\log n)$.  
Thus \Cref{thm:gaussian-cube-polytope-lb} shows that even constant-error approximation requires $\Omega(n)=\exp(\Omega(r_n^2))$ halfspaces.

\subsubsection{Information-Theoretic Preliminaries}
\label{subsubsec:cube-lb-info-theory}

Our argument is information-theoretic, and we first briefly recall the standard notions we will use. 
If $\by$ is a random vector with density $f$, its differential entropy is
\[
	h(\by):=-\int f(y)\ln f(y)\,dy\,,
\]
whenever this quantity is well-defined.  
If $\nu\ll\mu$, the relative entropy (or Kullback--Leibler (KL) divergence) of $\nu$ with respect to $\mu$ is 
\[
	\dkl(\nu\|\mu):=\int \ln\pbra{\frac{d\nu}{d\mu}}\,d\nu\,.
\]
We will also use Pinsker's inequality in the form
\[
	\dtv(\nu,\mu)\leq \sqrt{\frac{1}{2} \dkl(\nu\|\mu)}
\]
where $\dtv$ is the statistical (or total-variation) distance between $\nu$ and $\mu$. 
Finally, we will require the following immediate consequence of the generalized entropy-power inequality of Zamir and Feder~\cite{ZamirFeder93}:

\begin{lemma}[Zamir--Feder entropy inequality]
\label{lem:zamir-feder-projection-entropy}
	Let $\bx=(\bx_1,\ldots,\bx_n)$ have i.i.d.~coordinates with finite
	differential entropy.  
	Let $M:\R^n\to\R^d$ have orthonormal rows, so
	$MM^\top =I_d$.  
	Then
	\[
		h(M\bx)\geq d\,h(\bx_1)\,.
	\]
\end{lemma}

\begin{proof}
	Let $\bg=(\bg_1,\ldots,\bg_n)$ be a Gaussian vector with independent
	coordinates chosen so that $h(\bg_i)=h(\bx_i)$ for every $i$. 
	By \cite[Theorem~1]{ZamirFeder93} we have 
	\[
		h(M\bx)\geq h(M\bg)\,.
	\]
	Since the coordinates of $\bx$ are i.i.d., the coordinates of $\bg$ have a
	common variance, say $\sigma^2$.  Therefore, using $MM^\top =I_d$, $M\bg\sim N(0,\sigma^2 I_d)$. 
	Hence $h(M\bg)=d\,h(\bg_1)=d\,h(\bx_1)$, which proves the claim.
\end{proof}

\subsubsection{Proof of~\Cref{thm:gaussian-cube-polytope-lb}}

The proof starts with the observation that an intersection of $m$ halfspaces depends only on the projection onto the span of its defining normals. 
Thus, if such a polytope approximates the cube, this low-dimensional projection must distinguish the ambient Gaussian from the Gaussian conditioned on the cube. 
However, for the parameters in \Cref{thm:gaussian-cube-polytope-lb}, the entropy inequality above shows that no such projection can distinguish the two distributions well.

\begin{proof}[Proof of~\Cref{thm:gaussian-cube-polytope-lb}]
	Write
	\[
		P=\bigcap_{j=1}^m\cbra{x\in\R^n:a_j\cdot x\leq b_j}\,.
	\]
	Let $V:=\mathrm{span}(a_1,\ldots,a_m)$ and note that $d:=\dim (V)\leq m$. 
	We will write $\pi_V: \R^n \to V$ for the orthogonal projection onto $V$. 
	Since every defining inequality of $P$ depends only on $\pi_V x$, there is a set $A\sse V$ such that $P=\pi_V^{-1}(A)$. 

	Suppose $\bg\sim N(0,I_n)$. 
	Let $\gamma_V$ be the law of the random variable $\pi_V\bg$, and let $\nu_V$ be the law of $\pi_V \bg$ conditioned on the event $\bg \in C_n$. 
	We first show that approximating $C_n$ by $P$ would distinguish $\nu_V$ from $\gamma_V$. 
	Note that $\gamma_n(P) = \gamma_V(A)$ and that 
	\[
		\gamma_n(P\cap C_n) = \Prx_{\bg\sim N(0,I_n)}\sbra{\pi_V\bg\in A, \bg\in C_n} = \frac{1}{2}\nu_V(A)
	\]
	where we used the fact that $\gamma_n(C_n) = \frac{1}{2}$. 
	Therefore
	\begin{align}
		\dG(P, C_n) 
		= \gamma_n(P\triangle C_n)
		&=
		\gamma_n(P)+\gamma_n(C_n)-2\gamma_n(P\cap C_n) \nonumber \\
		&=
		\gamma_V(A)+\frac{1}{2}-\nu_V(A) \nonumber \\
		&=
		\frac{1}{2}-\pbra{\nu_V(A)-\gamma_V(A)} \nonumber \\
		&\geq
		\frac{1}{2}-\dtv(\nu_V,\gamma_V)\,. \label{eq:potato}
	\end{align}
	The remainder of the argument will upper bound $\dtv(\nu_V,\gamma_V)$. 

	Set $q:=\gamma_1([-r_n,r_n])$. 
	The normalization $\gamma_n(C_n)=1/2$ gives
	\[
		q^n=\frac{1}{2}\,,
		\qquad\text{so}\qquad
		\ln\pbra{\frac{1}{q}} = \frac{\ln 2}{n}\,.
	\]
	Let $\mu$ be the one-dimensional Gaussian conditioned on $[-r_n,r_n]$:
	\[
		d\mu(x)=\frac{\Indicator_{[-r_n,r_n]}(x)}{q}\,d\gamma_1(x)\,.
	\]
	Then
	\[
		\dkl(\mu\|\gamma_1)=\ln\pbra{\frac{1}{q}} =\frac{\ln 2}{n}\,.
	\]
	If $\bx\sim\gamma_n(\cdot\mid C_n)$, then $\bx=(\bx_1,\ldots,\bx_n)\sim\mu^{\otimes n}$.
	By symmetry, $\Ex\sbra{\bx_i}=0$.  
	Writing $\tau:=\Ex\sbra{\bx_i^2}$, note that  $\Ex\sbra{\bx\bx^\top}=\tau I_n$. 

	Identify $V$ with $\R^d$, and write $\pi_V\bx=M\bx$, where
	$M:\R^n\to\R^d$ has orthonormal rows.  By
	\Cref{lem:zamir-feder-projection-entropy},
	\begin{equation} \label{eq:zf-app}
		h(M\bx)\geq d\,h(\mu)\,.
	\end{equation}
	We now bound the relative entropy of $\nu_V$ with respect to $\gamma_V$.
	Recalling that $\nu_V$ is the law of $M\bx$ and that $\gamma_V$ is the standard Gaussian $\gamma_d$ on $\R^d$, we have 
	\begin{align*}
		\dkl\pbra{\nu_V\|\gamma_V}
		&=
		\frac{1}{2}\Ex\sbra{\|M\bx\|_2^2}
		+\frac d2\ln(2\pi)
		-h(M\bx) \\
		&\leq
		\frac{d\tau}{2}
		+\frac d2\ln(2\pi)
		-d\,h(\mu) \tag{\Cref{eq:zf-app}} \\
		&=
		d\pbra{
			\frac{\tau}{2}
			+\frac{1}{2}\ln(2\pi)
			-h(\mu)
		} \\
		&=
		d\,\dkl(\mu\|\gamma_1)
		=
		\frac{d\ln 2}{n}
		\leq
		\frac{m\ln 2}{n}\,.
	\end{align*}
	Here we used
	\[
		\Ex\sbra{\|M\bx\|_2^2}
		=
		\operatorname{tr}\pbra{M\Ex\sbra{\bx\bx^\top}M^\top}
		=
		\tau\operatorname{tr}(MM^\top)
		=
		d\tau\,.
	\]
	By Pinsker's inequality, 
	\[
		\dtv(\nu_V,\gamma_V)
		\leq
		\sqrt{\frac{1}{2}\dkl(\nu_V\|\gamma_V)}
		\leq
		\sqrt{\frac{m\ln 2}{2n}}\,.
	\]
	Combining this with \Cref{eq:potato} yields
	\[
		\dG(P, C_n) 
		\geq
		\frac{1}{2}-\sqrt{\frac{m\ln 2}{2n}}\,,
	\]
	which, after rearranging, completes the proof of \Cref{thm:gaussian-cube-polytope-lb}. 
\end{proof}

%% file: sections/technical-lemmas.tex
\section{Technical Lemmas}

We record the elementary analytic facts used above.
Throughout this appendix, absolutely continuous means absolutely continuous as a function of the parameter on $[0,1]$.

\begin{lemma}[Bounded continuous vector fields]
\label{lem:bounded-ode}
	Let $F:[0,1]\times\R^N\to\R^N$ be continuous and bounded, where $\R^N$ is equipped with the Euclidean norm.
	For every initial value $\bar{a}_0\in\R^N$, the initial-value problem
	\[
		\bar{a}'(\theta)
		=
		F(\theta,\bar{a}(\theta))\,,
		\qquad
		\bar{a}(0)=\bar{a}_0
	\]
	has an absolutely continuous solution on $[0,1]$.
	Moreover, if $\|F\|_\infty\leq  M$, then every solution satisfies
	\[
		\|\bar{a}(\theta)-\bar{a}(\theta')\|_2
		\leq 
		M|\theta-\theta'|
	\]
	for all $\theta,\theta'\in[0,1]$, and hence $\|\bar{a}'(\theta)\|_2\leq  M$ for a.e.~$\theta\in[0,1]$.
\end{lemma}

\begin{proof}
	Peano's existence theorem~\cite[Theorem 2.19]{Tes12} gives a local absolutely continuous solution.
	Since the vector field is bounded by $M$, every local solution is $M$-Lipschitz.
	In particular, a solution cannot run off to infinity in finite time, so the local solution extends to all of $[0,1]$.

	The Lipschitz estimate follows directly by integrating the differential equation.
	Indeed, for any solution and any $\theta,\theta'\in[0,1]$,
	\[
		\|\bar{a}(\theta)-\bar{a}(\theta')\|_2
		\leq 
		\int_{\min\{\theta,\theta'\}}^{\max\{\theta,\theta'\}}
		\|\bar{a}'(s)\|_2\,ds
		\leq 
		M|\theta-\theta'|\,.
	\]
	The a.e.~bound $\|\bar{a}'(\theta)\|_2\leq  M$ follows from $\bar{a}'(\theta)=F(\theta,\bar{a}(\theta))$ for a.e.~$\theta$.
\end{proof}

The next lemma justifies differentiating the softmax interpolation as an $L^2$-valued path in the proof of \Cref{thm:GP-sparsification}. 

\begin{lemma}[Softmax paths as $L^2$-valued absolutely continuous curves]
\label{lem:l2-chain}
	Let $T$ be finite.
	Suppose $v_t:[0,1]\to\R^n$ and $a_t:[0,1]\to\R$ are absolutely continuous and that, for constants $L,M<\infty$,
	\[
		\|v_t'(\theta)\|_2\leq  L\,,
		\qquad
		|a_t'(\theta)|\leq  M
	\]
	for every $t\in T$ and for a.e.~$\theta\in[0,1]$.
	Fix $\beta>0$, and set
	\[
		H_\theta(g)
		=
		\frac{1}{\beta}
		\log\pbra{
		\sum_{t\in T}
		\exp\pbra{\beta\pbra{g\cdot v_t(\theta)+a_t(\theta)}}
		}\,.
	\]
	Then $\theta\mapsto H_\theta$ is absolutely continuous as a map from $[0,1]$ into $L^2$.
	Moreover, for a.e.~$\theta$, the $L^2$-derivative is represented by
	\[
		\frac{d}{d\theta}H_\theta(g)
		=
		\sum_{t\in T}
		\omega_t^\theta(g)
		\pbra{
		g\cdot v_t'(\theta)
		+
		a_t'(\theta)
		}\,,
	\]
	where
	\[
		\omega_t^\theta(g)
		:=
		\frac{
		\exp\left(\beta(g\cdot v_t(\theta)+a_t(\theta))\right)
		}{
		\sum_{t'\in T}
		\exp\left(\beta(g\cdot v_{t'}(\theta)+a_{t'}(\theta))\right)
		}\,.
	\]
\end{lemma}

\begin{proof}
	For each fixed $g$, the derivative formula follows from the ordinary chain rule for absolutely continuous functions, since log-sum-exp is smooth.
	Moreover, because the $\omega_t^\theta(g)$'s are non-negative and sum to one,
	\[
		\left|
		\frac{d}{d\theta}H_\theta(g)
		\right|
		\leq 
		\sum_{t\in T}
		\omega_t^\theta(g)
		\pbra{|g\cdot v_t'(\theta)|+|a_t'(\theta)|}
		\leq 
		L\|g\|_2+M\,.
	\]
	The bound $L\|g\|_2+M$ is in $L^2(\gamma_n)$ and is independent of $\theta$. 
	Therefore the difference quotients are dominated in $L^2$, the pointwise derivatives are Bochner integrable, and dominated convergence allows us to pass from pointwise absolute continuity to absolute continuity as an $L^2$-valued curve. 
	Thus, for every $0\leq \theta_1\leq \theta_2\leq 1$,
	\[
		H_{\theta_2}-H_{\theta_1}
		=
		\int_{\theta_1}^{\theta_2}
		\frac{d}{d\theta}H_\theta\,d\theta\,.
	\]
	This proves both absolute continuity and the claimed derivative formula. 
\end{proof}